\documentclass[preprint, 11pt]{elsarticle}
\usepackage{lineno}
\pdfoutput=1

\usepackage{amsmath}
\usepackage{amssymb}
\usepackage{mathtools}
\usepackage{mathrsfs}
\usepackage{color}
\usepackage{xcolor}
\usepackage{graphicx}
\usepackage[percent]{overpic}
\usepackage{url}
\usepackage{caption}
\usepackage{subcaption}
\usepackage{enumerate}
\usepackage{overpic}
\usepackage{amsthm}
\usepackage{moreverb}
\usepackage{multirow}
\usepackage{multicol}
\usepackage[pdftex,colorlinks,bookmarksopen,bookmarksnumbered,citecolor=red,urlcolor=red]{hyperref}
\usepackage{mwe}
\usepackage{graphbox}
\usepackage{textcomp}
\usepackage{tabularx}
\usepackage{algorithm}
\usepackage{algorithmicx}
\usepackage{algpseudocode}
\usepackage{tikz}
\newtheorem{theorem}{Theorem}[section]
\newtheorem{lemma}[theorem]{Lemma}

\def\b{{\bm b}}

\def\x{{\bm x}}

\def\0{\boldsymbol{0}}

\def\calsf{\mathbf{\mathcal{S}}}
\def\caldf{\mathbf{\mathcal{D}}}

\def\calvf{\mathbf{\mathcal{V}}}
\def\calif{\mathbf{\mathcal{I}}}

\def\calmf{\mathbf{\mathcal{M}}}
\def\calpf{\mathbf{\mathcal{P}}}

\def\caltf{\mathbf{\mathcal{T}}}
\def\calkf{\mathbf{\mathcal{K}}}
\def\calwf{\mathbf{\mathcal{W}}}
\def\calaf{\mathbf{\mathcal{A}}}
\def\calbf{\mathbf{\mathcal{B}}}

\def \bV{{\mathbf V}}

\def \b0{{\mathbf 0}}

\newcommand{\bm}[1]{\mbox{\boldmath{$#1$}}}
\DeclarePairedDelimiter\abs{\lvert}{\rvert}%
\newcommand\dual[1]{\left\langle#1\right\rangle}%
\definecolor{ForestGreen}{RGB}{34,139,34}
\newcommand\norm[1]{\left\lVert#1\right\rVert}

\usepackage{verbatim}
\usepackage{graphicx}
\usepackage{epstopdf}

\newtheorem*{remark}{Remark}

\begin{document}

\begin{frontmatter}

\title{Schur complement domain decomposition for 3D fluid - 2D plate interaction system: analysis and preconditioning}





\author{Lander Besabe}
\ead{lbesabe@clemson.edu}

\author{Hyesuk Lee}
\ead{hklee@clemson.edu}

\address{School of Mathematical and Statistical Sciences, Clemson University, Clemson, SC 29634-0975, USA}

\begin{abstract}

Fluid–structure interaction (FSI) problems involving three-dimensional fluids coupled with thin elastic plates arise in many engineering applications and present significant computational challenges due to the strong coupling across the fluid–structure interface. In this work, we develop a Schur complement domain decomposition method for a 3D fluid–2D plate interaction problem in which the fluid is modeled by the unsteady Stokes equations and the structure by a reformulated Kirchhoff plate model. Starting from a mixed finite element discretization with a Lagrange multiplier (LM) enforcing the interface conditions, we derive an interface Schur complement formulation that decouples the fluid and plate subproblems while preserving the strong interface coupling without requiring subiterations. We then analyze the conditioning of the resulting Schur complement system matrix and show how its condition number deteriorates under mesh refinement. Motivated by this analysis, we introduce an interface-based preconditioner and establish theoretical bounds showing that the proposed preconditioner significantly improves the conditioning of the resulting system. Numerical experiments verify the expected spatial and temporal convergence rates, demonstrate significant reductions in condition numbers and GMRES iteration counts for the proposed preconditioner compared with both the unpreconditioned system and a block Jacobi preconditioner, and illustrate the robustness of the proposed method for challenging added-mass regimes.

\end{abstract}

\begin{keyword}
Fluid-structure interaction \sep
Mixed finite element methods \sep
Lagrange multiplier methods \sep
Stokes equations \sep
Kirchhoff-Love plate \sep
Domain decomposition \sep
Schur complement method
\end{keyword}

\end{frontmatter}

\section{Introduction} \label{sec:intro}

Fluid-structure interaction (FSI) problems couple fluid and structural dynamics through interface conditions that enforce continuity of velocity and stress force. The numerical simulation of such systems is computationally demanding due to the strong coupling between the governing equations, the need to accurately enforce the interface conditions \cite{Causin-Nobile2005, Richter2017}, and the large computational cost. Such problems arise in a variety of scientific and engineering settings, including hemodynamics \cite{Quaini2011, Duca2025, Quarteroni2000} and aeroelasticity \cite{Zheng2023, Svacek2008}. In many applications, the structural component is thin relative to the surrounding fluid domain and may be represented by a lower-dimensional model, leading to coupled systems consisting of a three-dimensional fluid and a two-dimensional elastic structure. Such formulations provide an efficient description of the underlying physics while introducing additional challenges in the treatment of the fluid-structure interface \cite{BesabeLee2026}. Consequently, the development of stable and efficient numerical methods for 3D fluid-2D plate interaction problems continues to be an important area of research, particularly in the development of partitioned approaches capable of handling strong interface coupling.

Domain decomposition methods have long been employed as an efficient framework for the numerical simulation of partial differential equations (PDEs) by decomposing a global problem into smaller subproblems coupled through interface conditions. The earliest developments in domain decomposition techniques \cite{QuarteroniValli1999, ToselliWidlund2005} laid the foundations for interface-based methods and Schur complement formulations through Steklov-Poincar\'{e} operators. In the context of FSI problems, these techniques translate to treating the fluid and structure subproblems independently on their respective domains, while enforcing coupling conditions on the interface. In \cite{Deparis2006}, a partitioned approach for the FSI problem was developed through the introduction of Steklov-Poincar\'{e} operators and interface equations. Subsequent developments led to the construction of Robin-Robin and Krylov-based interface solvers \cite{Badia2009, Nobile2008}, as well as more recent Schur-complement formulations that enable the decoupling of monolithic FSI systems while preserving strong interface coupling \cite{de_Castro2025}.

Beyond Schur complement formulations, a broad class of domain decomposition techniques has been developed for partitioned FSI simulations. Classical Dirichlet-Neumann iterations \cite{Kuttler2006} remain among the most widely used coupling strategies due to their modularity, although they are known to suffer from the added-mass effect for incompressible flows \cite{Causin-Nobile2005}. To improve robustness and convergence, Robin transmission conditions \cite{Badia2008}, optimized Robin-Robin procedures \cite{GerardoGiorda2010}, and interface quasi-Newton methods \cite{Spenke2023} have been proposed, yielding substantially improved convergence properties and greater flexibility in the treatment of interface conditions.

Compared to classical FSI problems involving fluid and structure of the same dimension, the literature on numerical methods for 3D fluid–2D plate interaction systems remains relatively limited. Earliest works on 3D fluid-2D plate FSI \cite{Avalos2014, Cheng2008} employed monolithic formulations, requiring $H^2$-conforming discretization for the plate displacement due to the presence of the biharmonic operator. More recently, attention has expanded to reduced-dimensional plate models coupled with viscous fluids. For example, in \cite{Brandt2026}, a finite element method for the interaction of an incompressible viscous fluid with a fully averaged poroelastic Kirchhoff plate was developed and analyzed. Other recent works \cite{Geredeli_Kunwar_Lee2024, Geredeli2026, Avalos2026} explored partitioned approaches for coupling the fluid and plate subproblems  using fixed point iterations at each time step. In previous work \cite{BesabeLee2026}, we utilized the same partitioned algorithm for a linear 3D fluid-2D plate interaction problem by reformulating the fourth-order plate equation as a coupled system of second-order PDEs. This approach eliminated the need for $H^2$-conforming elements, allowing the use of standard $C^0$-conforming finite element spaces and providing greater flexibility in the construction of the finite element discretization.

In this work, we develop a Schur complement domain decomposition method for the numerical approximation of a 3D fluid-2D plate interaction problem. Starting from a coupled finite element formulation, we derive an interface Schur complement system that enables the fluid and structural subproblems to be solved independently without subiterations while preserving the strong interface coupling. We establish analytical estimates for the condition number of the resulting Schur complement system matrix and show that it may deteriorate with mesh refinement. To address this issue, we introduce a preconditioner based on the structure of the interface operator and derive theoretical bounds for the condition number of the preconditioned system.

The remainder of the paper is organized as follows. In Sec.~\ref{sec:gov_eq}, we briefly recall the 3D fluid-2D plate FSI model and its reformulation. We present the fully discrete problem and the partitioned approach through the Schur complement system in Sec.~\ref{sec:fully-disc}. We proceed by examining the condition number of the Schur complement system matrix in Sec.~\ref{sec:cond_num_schurc}, along with a proposed preconditioner, and derive the resulting condition number of the preconditioned system. In Sec.~\ref{sec:num_res}, two numerical tests are provided, which demonstrate the validity and robustness of our technique. We draw some conclusions in Sec.~\ref{sec:conclusions}.

\section{Governing Equations}\label{sec:gov_eq}

We consider a fluid in a bounded domain $\Omega_f\subset\mathbb{R}^3$ with sufficiently smooth boundary $\partial\Omega_f = \bar{S}\cup\bar{\Omega}_p$, with $\Omega_p\cap S = \emptyset$ where $S$ is the rigid portion of the boundary and $\Omega_p$ is the portion supporting the elastic plate. We assume that $\Omega_p$ is a fixed planar surface (not necessarily horizontal), embedded in $\mathbb{R}^3$.

The dynamics of the FSI system is governed by the unsteady Stokes equations in $\Omega_f$ and either the "Euler-Bernoulli" (no rotation) or "Kirchhoff" (with rotation) equations in $\Omega_p$. The FSI problem reads: for a final time $T>0$, find velocity $\bm{u} = (u_1, u_2, u_3)^T:\Omega_f\times(0,T)\rightarrow\mathbb{R}^3$, pressure $p:\Omega_f\times(0,T)\rightarrow\mathbb{R}$, and plate displacement in the normal direction $w:\Omega_p\times(0,T)\rightarrow\mathbb{R}$ such that
\begin{align} \label{eq:momentum}
    &\rho_f\partial_t\bm{u} - 2\nu_f~ \nabla \cdot D(\bm{u}) + \nabla p = \bm{f} &&\text{in }\Omega_f\times(0,T), \\
    &\nabla\cdot\bm{u} = 0, &&\text{in }\Omega_f\times(0,T),\label{eq:continuity} \\
    &\bm{u} = \bm{0} &&\text{on }S\times(0,T),\label{eq:fluid_bc} \\
    &\rho_p\partial_{tt}w - \rho\partial_{tt}\Delta w + \Delta^2 w = -(\bm{\sigma}_f\bm{n})\cdot\bm{n} &&\text{in }\Omega_p\times(0,T), \label{eq:plate_eq} \\
    &w = \Delta w = 0 &&\text{on }\partial\Omega_p\times(0,T), \label{eq:hingedbc} \\
    &\bm{u} = \partial_t{w}~\bm{n} &&\text{on }\Omega_p\times(0,T), \label{eq:interface}
\end{align}
where $\partial_t \bm{u} = \frac{\partial \bm{u}}{\partial t}$, $\partial_{tt}w = \frac{\partial^2 w}{\partial t^2}$, $\rho_f$ is the fluid density, $\nu_f$ denotes the kinematic viscosity of the fluid, $\rho_p$ is the plate surface density, $\rho$ is the rotational inertia parameter, $\bm{n}$ is the outward unit normal vector to $\Omega_f$, and $D(\cdot)$ is the strain rate tensor given by $D(\bm{v}) = \frac{1}{2}(\nabla\bm{v} + \nabla\bm{v}^T)$. The fluid Cauchy stress tensor is given by
\begin{equation} \label{eq:stress_tensor}
    \bm{\sigma}_f = 2\nu_f D(\bm{u}) - p\bm{I}.
\end{equation}
The system is supplemented with the initial conditions
\begin{equation} \label{eq:IC}
    (\bm{u}, w, \partial_t w) = (\bm{u}_0, w_0, w_{t0}).
\end{equation}

In much of the existing literature on 3D fluid - 2D plate interaction, e.g., \cite{Chueshov2013, Avalos2014, Geredeli_Kunwar_Lee2024, Brandt2026}, the plate is assumed, without loss of generality, to occupy the horizontal plane $z=0$. Under this assumption, the outward unit normal to the interface is simply $\bm{n}=\bm{e}_3=(0,0,1)^T$, and the coupling conditions are written in terms of the vertical component of the fluid velocity and stress, such as $\bm{u}=\partial_t w\bm{e}_3$ and the normal traction $(\bm{\sigma}_f \bm{e}_3)\cdot \bm{e}_3$. In contrast, we formulate the problem in a more general setting by allowing the plate to occupy an arbitrary fixed planar surface in $\mathbb{R}^3$. Consequently, the interface conditions are expressed using the outward unit normal vector $\bm{n}$, namely $\bm{u}=\partial_t w\bm{n}$ and the plate load $-(\bm{\sigma}_f\bm{n})\cdot\bm{n}$.

Note that in a number of previous works \cite{Geredeli_Kunwar_Lee2024, Avalos2014, Avalos2026}, the plate displacement $w$ is assumed to be in $H^2(\Omega_p)$ in the weak formulation. Following \cite{BesabeLee2026}, we reformulate the fourth order plate equations \eqref{eq:plate_eq}-\eqref{eq:hingedbc} into coupled second order equations by introducing an auxiliary variable $z$ such that 
\begin{equation}\label{eq:poisson_variable}
    z = -\Delta w,
\end{equation}
which corresponds to the plate curvature (or bending variable). This transforms \eqref{eq:plate_eq}-\eqref{eq:hingedbc} into the system of second-order PDEs given by
\begin{align}
    &\rho_p\partial_{tt}w + \rho \partial_{tt}z - \Delta z = -(\bm{\sigma}_f\bm{n})\cdot\bm{n} &&\text{in }\Omega_p\times(0,T), \label{eq:plate_eq_order2}\\
    &z + \Delta w = 0 &&\text{in }\Omega_p\times(0,T),\label{eq:plate_poisson}\\
    &w = z = 0 &&\text{on }\partial\Omega_p\times(0,T). \label{eq:clapmedbc2}
\end{align}
The reformulation requires additional initial values for $z$ and $\partial_t z$ which are determined from the original initial data by
\begin{equation*}
    z_0 = -\Delta w_0, \quad z_{t0} = -\Delta w_{t0}.
\end{equation*}
Following \cite{Geredeli_Kunwar_Lee2024}, we begin by defining the following function spaces
\begin{align*}
    &\bm{U} = \{\bm{v} = (v_1, v_2, v_3)^T\in [H^1(\Omega_f)]^3: \bm{v} = \bm{0} \text{ on }S\},\\
    &Q = L^2(\Omega_f), \\ 
    &W = \left\{\varphi\in H^1(\Omega_p):\varphi = 0 \text{ on }\partial\Omega_p\right\} = H^1_0 (\Omega_p).
\end{align*}
It is also straightforward to show that the mean plate velocity $\partial_t w$ has mean zero due to the divergence theorem and the incompressibility of the fluid:
\begin{equation} \label{eq:wdot_zeromean}
    0 = \int_{\Omega_f} \nabla\cdot\bm{u}\,d\Omega_f = \int_S \bm{u}\cdot \bm{n}_S\,dS + \int_{\Omega_p} \bm{u}\cdot\bm{n}\,d\partial\Omega_p = \int_{\Omega_p} \partial_t w \,d\Omega_p,
\end{equation}
where $\bm{n}_S$ is the unit outward normal vector to the surface $S$. For a more detailed discussion of the model, see, for example,  \cite{Avalos2014, Chueshov2013, Chambolle2005}.

To enforce the coupling conditions, we introduce a Lagrange multiplier $\bm{g}\in \bm{G} = [H^{-1/2}(\Omega_p)]^3$ representing the interface traction
\begin{equation} \label{eq:LM_defn}
    \bm{g} \coloneqq \bm{\sigma}_f\bm{n} \quad \text{on }\Omega_p\times(0, T).
\end{equation}

Let $(\cdot,\cdot)_\gamma$ be the $L^2$ inner product over domain $\gamma$, $\norm{\cdot}_\gamma$ be the corresponding norm, and $\dual{\cdot, \cdot}_\gamma$ be the dual pairing between $H^{1/2}(\gamma)$ and $H^{-1/2}(\gamma)$. We also denote the $H^1(\gamma)$ Sobolev norm of $f$ by $\norm{f}_{1, \gamma}^2 = \norm{f}_{\gamma}^2 + \norm{\nabla f}_{\gamma}^2$. The variational problem associated to \eqref{eq:momentum}-\eqref{eq:fluid_bc}, \eqref{eq:interface} and \eqref{eq:plate_eq_order2}-\eqref{eq:clapmedbc2} reads: find $(\bm{u}, p, w, z, \bm{g})\in \bm{U}\times Q\times W\times W\times \bm{G}$ such that
\begin{align} \label{eq:momentum_weak}
    &\rho_f(\partial_t\bm{u}, \bm{v})_{\Omega_f} + 2\nu_f(D(\bm{u}), D(\bm{v}))_{\Omega_f} - (p, \nabla\cdot\bm{v})_{\Omega_f} - \dual{\bm{g}, \bm{v}}_{\Omega_p} = (\bm{f}, \bm{v})_{\Omega_f} && \forall \bm{v}\in\bm{U},\\
    &(q, \nabla\cdot \bm{u})_{\Omega_f} = 0 &&\forall q\in Q, \label{eq:continuity_weak} \\
    & \rho_p(\partial_{tt}w, \eta)_{\Omega_p} + \rho(\partial_{tt}z, \eta)_{\Omega_p} + (\nabla z, \nabla \eta)_{\Omega_p} + \dual{\bm{g}\cdot\bm{n}, \eta}_{\Omega_p} = 0 &&\forall\eta\in W, \label{eq:plate_weak}\\
    &(\nabla w, \nabla \varphi)_{\Omega_p} - (z, \varphi)_{\Omega_p} = 0 && \forall\varphi\in W, \label{eq:poisson_weak}\\
    & \dual{\partial_t w \bm{n} - \bm{u}, \bm{\lambda}}_{\Omega_p} = 0 &&\forall \bm{\lambda}\in \bm{G}.\label{eq:interface_weak}
\end{align}

\section{Fully discrete problem} \label{sec:fully-disc}

In this section, we present a fully discrete formulation of \eqref{eq:momentum_weak}-\eqref{eq:interface_weak}. We then reformulate the coupled system using a Schur complement approach, which forms the basis of the proposed partitioned algorithm. Since the analysis of the coupled formulation, for both the semi- and fully-discrete problems, follows directly from \cite{BesabeLee2026}, we omit those details here and focus on the aspects needed to develop the domain decomposition approach.

We discretize the time interval $[0,T]$ into $N\in\mathbb{Z}^+$ uniform time steps of size $\delta t = T/N$ and denote $t^n = n\delta t$. For any generic  quantity $f$, we denote its approximation at time $t^n$ by $f^n$. We use the following first-order backward Euler approximation:
\begin{equation} \label{eq:back_euler}
    \partial_{t}\bm{u}\approx \dot{\bm{u}}^{n+1}\coloneqq\frac{\bm{u}^{n+1}-\bm{u}^{n}}{\delta t}, \quad \partial_{tt}w\approx \ddot{w}^{n+1}\coloneqq\frac{\dot{w}^{n+1}-\dot{w}^n}{\delta t} = \frac{w^{n+1} - 2w^n + w^{n-1}}{(\delta t)^2}.
\end{equation}
Applying these approximations to \eqref{eq:momentum_weak}-\eqref{eq:interface_weak} yields, at each time instant $t^n$, the following saddle-point problem which reads: find $\left(\bm{u}^{n+1}, w^{n+1}, z^{n+1}, p^{n+1}, \bm{g}^{n+1}\right)\in \bm{U}\times W\times W\times Q\times \bm{G}$ such that
\begin{equation} \label{eq:saddle_point_prob}
    \begin{aligned}
        a\left((\bm{u}^{n+1}, w^{n+1}, z^{n+1}), (\bm{v}, \varphi, \eta)\right) + b\left((\bm{v}, \eta), \left(\delta t p^{n+1}, \delta t \bm{g}^{n+1}\right)\right) &= f_1(\bm{v}, \varphi, \eta), \\
        b\left(\left(\bm{u}^{n+1}, \frac{1}{\delta t}w^{n+1}\right),(q, \bm{\lambda})\right) &= f_2 (q, \bm{\lambda}),
    \end{aligned}
\end{equation}
where we used the fact that $\dual{\bm{g}^{n+1}\cdot\bm{n}, \eta}_{\Omega_p} = \dual{\bm{g}^{n+1}, \eta~\bm{n}}_{\Omega_p}$, $a: (\bm{U}\times W\times W) \times (\bm{U}\times W\times W) \to \mathbb{R}$, $b:(\bm{U}\times W)\times (Q\times \bm{G}) \to \mathbb{R}$, $f_1:\bm{U}\times W\times W \to \mathbb{R}$, 
and $f_2:Q\times \bm{G} \to \mathbb{R}$ are defined by
\begin{align*}
    &\begin{aligned}
        a&\left((\bm{u}, w, z), (\bm{v}, \varphi, \eta) \right) \coloneqq \rho_f(\bm{u}, \bm{v})_{\Omega_f} + 2\nu_f\delta t(D(\bm{u}), D(\bm{v}))_{\Omega_f} + \frac{\rho_p}{\delta t}(w, \eta)_{\Omega_p} + \frac{\rho}{\delta t}(z, \eta)_{\Omega_p} \\
        & + \delta t(\nabla z, \nabla\eta)_{\Omega_p} - \frac{\rho_p}{\delta t}(z, \varphi)_{\Omega_p} + \frac{\rho_p}{\delta t}(\nabla w, \nabla \varphi)_{\Omega_p},
    \end{aligned}\\
    & b\left((\bm{v}, \eta), (q, \bm{\lambda})\right) \coloneqq -(\nabla\cdot \bm{v}, q)_{\Omega_f} + \dual{-\bm{v} + \eta~\bm{n}, \bm{\lambda}}_{\Omega_p}, 
     \\
    &\begin{aligned}
        f_1&(\bm{v}, \varphi, \eta) \coloneqq \delta t(\bm{f}^{n+1}, \bm{v})_{\Omega_f} + \rho_f (\bm{u}^n, \bm{v})_{\Omega_f} + \frac{2\rho_p}{\delta t}(w^n, \eta)_{\Omega_p} - \frac{\rho_p}{\delta t}(w^{n-1}, \eta)_{\Omega_p} \\
        & + \frac{2\rho}{\delta t} (z^n, \eta)_{\Omega_p} - \frac{\rho}{\delta t} (z^{n-1}, \eta)_{\Omega_p},
    \end{aligned}\\
    & f_2(q, \bm{\lambda}) \coloneqq  \frac{1}{\delta t}\dual{w^{n}~\bm{n}, \bm{\lambda}}_{\Omega_p}.
\end{align*}

\begin{remark}
    Existence, uniqueness, and stability of the semi-discrete problem \eqref{eq:saddle_point_prob} follow directly from the analysis in \cite{BesabeLee2026}; hence, we omit the proof. The extension from a scalar Lagrange multiplier to a vector-valued Lagrange multiplier is straightforward, as the previous analysis does not rely on the multiplier being scalar. In particular, the only modification is to replace the vertical component of the interface traction with the full traction vector, and all arguments carry over without further changes.
\end{remark}
Suppose the computational domains, $\Omega_f$ and $\Omega_p$, have Lipschitz and polytopal boundaries, respectively. Let $\mathcal{T}_{h_f}$ and $\mathcal{T}_{h_p}$ be quasi-uniform, shape regular triangulations of $\overline{\Omega}_f$ and $\overline{\Omega}_p$, respectively, into tetrahedra where $h_f$ and $h_p$ are the corresponding mesh sizes, defined by
\begin{equation*}
    h_f = \max_{E\in \mathcal{T}_{h_f}}\mathrm{diam}(E) \quad\text{ and }\quad h_p = \max_{K\in \mathcal{T}_{h_p}}\mathrm{diam}(K).
\end{equation*}
We define the following conforming finite element spaces:
\begin{equation*}
    \begin{aligned}
        &\bm{U}_h = \{\bm{v}_h\in \bm{U}: \bm{v}_h|_{E}\in [\mathbb{P}_{k_u}(E)]^3, \quad\forall E\in \mathcal{T}_{h_f}\},\\
        &Q_{h} = \{q_h\in Q: q_h|_{E}\in \mathbb{P}_{k_p}(E), \quad \forall E\in \mathcal{T}_{h_f}\}, \\
        &W_h = \{\varphi_h\in W: \varphi_h|_{K}\in \mathbb{P}_{k_w}(K), \quad \forall K\in \mathcal{T}_{h_p}\}, \\
        &\bm{G}_{h} = \{\lambda_h\in G: \lambda_h|_{K}\in [\mathbb{P}_{k_g}(K)]^3, \quad \forall K\in \mathcal{T}_{h_p}\},
    \end{aligned}
\end{equation*}
for a fixed $1\leq k_\star\in \mathbb{Z}^+$ and $\mathbb{P}_{k_\star}$ is the space of polynomials with degree $\leq k_\star$. 

The fully-discrete problem reads: find $\left(\bm{u}_h^{n+1}, w_h^{n+1}, z_h^{n+1}, p_h^{n+1}, \bm{g}_h^{n+1}\right)\in \bm{U}_h\times W_h\times W_h\times Q_h\times \bm{G}_h$ such that
\begin{align} \label{eq:momentum_full_disc_weak}
    &\begin{aligned}
    &\rho_f(\bm{u}_h^{n+1}, \bm{v}_h)_{\Omega_f} + 2\nu_f\delta t(D(\bm{u}_h^{n+1}), D(\bm{v}_h))_{\Omega_f} - \delta t(p_h^{n+1}, \nabla\cdot \bm{v}_h)_{\Omega_f} \\
    &\quad - \delta t\dual{\bm{g}_h^{n+1}, \bm{v}_h}_{\Omega_p} = \delta t(\bm{f}^{n+1}, \bm{v})_{\Omega_f} + \rho_f (\bm{u}_h^n, \bm{v})_{\Omega_f}
    \end{aligned}
     &&\forall \bm{v}_h\in \bm{U}_h,\\
    &(\nabla\cdot \bm{u}_h^{n+1}, q_h)_{\Omega_f} = 0 &&\forall q_h\in Q_h,\label{eq:continuity_full_disc_weak} \\
    &\begin{aligned}
    &\frac{\rho_p}{(\delta t)^2}(w_h^{n+1}, \eta_h)_{\Omega_p} + \frac{\rho}{(\delta t)^2}(z_h^{n+1}, \eta_h)_{\Omega_p} + (\nabla z_h^{n+1}, \nabla \eta_h)_{\Omega_p} + \dual{\bm{g}_h^{n+1}, \eta_h\bm{n}}_{\Omega_p} \\
    & \quad = \frac{2\rho_p}{(\delta t)^2}(w_h^n, \eta_h)_{\Omega_p} - \frac{\rho_p}{(\delta t)^2}(w_h^{n-1}, \eta_h)_{\Omega_p} + \frac{2\rho}{(\delta t)^2} (z_h^n, \eta_h)_{\Omega_p} - \frac{\rho}{(\delta t)^2} (z_h^{n-1}, \eta_h)_{\Omega_p} 
    \end{aligned}
    &&\forall \eta_h\in W_h, \label{eq:plate_full_disc_weak}\\
    & (z_h^{n+1}, \varphi_h)_{\Omega_p} = (\nabla w_h^{n+1}, \nabla \varphi_h)_{\Omega_p} &&\forall \varphi_h\in W_h, \label{eq:poisson_full_semi_disc_weak} \\
    & \dual{w_h^{n+1}~\bm{n}, \bm{\lambda}_h}_{\Omega_p} - \delta t\dual{\bm{u}_h^{n+1}, \bm{\lambda}_h}_{\Omega_p} =\dual{w_h^{n} \bm{n}, \bm{\lambda}_h}_{\Omega_p} &&\forall \bm{\lambda}_h \in \bm{G}_h. \label{eq:interface_full_disc_weak}
\end{align}

\begin{lemma}\label{lem:infsup}
    There exists a positive constant $\beta$ such that 
    \begin{equation}\label{eq:full_disc_infsup}
        \sup_{(\bm{v}_h, \eta_h)\in(\bm{U}_h, W_h)} \frac{b((\bm{v}_h, \eta_h), (q_h, \bm{\lambda}_h))}{\left(\norm{\bm{v}_h}_{1, \Omega_f}^2 + \norm{\eta_h}_{1, \Omega_p}^2\right)^{1/2}} \geq \beta \left(\norm{q_h}_{\Omega_f}^2 + \norm{\bm{\lambda}_h}_{-1/2, {\Omega}_p}\right)^{1/2},
    \end{equation}
    for any $(q_h, \bm{\lambda}_h)\in (Q_h, \bm{G}_h)$.
\end{lemma}

The proof of Lemma~\ref{lem:infsup} follows the same structure as Theorem 4.1 in \cite{BesabeLee2026} with specifc differences in the interface coupling which now involves the trace of $\bm{v}_h\in \bm{U}_h$ instead of only its third component. Similarly, the plate contribution acts only on its normal component using the identity $\dual{\bm{\lambda}_h\cdot\bm{n}, \eta}_{\Omega_p} = \dual{\bm{\lambda}_h, \eta~\bm{n}}_{\Omega_p}$.

\subsection{Schur complement system}

Let $\{\bm{\varphi}_i\}, \{\eta_i\}, \{\psi_i\}$, and $\{\bm{\mu}_i\}$ be the finite element (FE) basis functions for $\bm{U}_h$, $W_h$, $Q_h$, and $\bm{G}_h$, respectively. 
Then, system \eqref{eq:momentum_full_disc_weak}-\eqref{eq:interface_full_disc_weak} may be rewritten into the linear system
\begin{equation} \label{eq:fe_lin_system}
    \left[\begin{array}{ccccc}
        \calmf_f + \delta t \calkf_f & \bm{0} & \bm{0} & -\delta t \calpf^T & -\delta t \caltf^T \\
        \bm{0} & \frac{\rho_p}{(\delta t)^2}\calmf_p & \frac{\rho}{(\delta t)^2}\calmf_p + \calkf_p & \bm{0} & \calvf^T \\
        \bm{0} & -\calkf_p & \calmf_p & \bm{0} & \bm{0} \\
        -\delta t\calpf & \bm{0} & \bm{0} & \bm{0} & \bm{0} \\
        - \delta t \caltf &\calvf & \bm{0} & \bm{0} & \bm{0}
        \end{array}\right]\left[\begin{array}{c}
        \overline{\bm{u}}_h^{n+1} \\
        \overline{\bm{w}}_h^{n+1} \\
        \overline{\bm{z}}_h^{n+1} \\
        \overline{\bm{p}}_h^{n+1}\\
        \overline{\bm{g}}_h^{n+1}
    \end{array}\right] = \left[\begin{array}{c}
        \overline{\bm{f}}_1^{n+1} \\
        \overline{\bm{f}}_2^{n+1} \\
        \overline{\bm{f}}_3^{n+1} \\
        \overline{\bm{f}}_4^{n+1} \\
        \overline{\bm{f}}_5^{n+1}
    \end{array}\right],
\end{equation}
where the matrices in the system matrix are defined by
\begin{equation} \label{eq:fe_matrices}
    \begin{aligned}
    (\calmf_f)_{ij} = \rho_f(\bm{\varphi}_j, \bm{\varphi}_i)_{\Omega_f}, \quad (\calkf_f)_{ij} = 2\nu_f(D(\bm{\varphi}_j), D(\bm{\varphi}_i))_{\Omega_f}, \quad (\calpf)_{ij} = (\nabla\cdot\bm{\varphi}_j, \psi_i)_{\Omega_f},\\
    (\caltf)_{ij} = \dual{\bm{\varphi}_j, \bm{\mu}_i}_{\Omega_p}, \quad (\calmf_p)_{ij} = (\eta_j, \eta_i)_{\Omega_p}, \quad (\calkf_p)_{ij} = (\nabla\eta_j, \nabla\eta_i)_{\Omega_p}, \quad (\calvf)_{ij} = \dual{\eta_j\bm{n}, \bm{\mu}_i}_{\Omega_p}.
    \end{aligned}
\end{equation}
and the RHS components are defined by
\begin{equation} \label{eq:rhs_vec}
    \begin{aligned}
        (\overline{\bm{f}}_1^{n+1})_j &= \delta t(\bm{f}^{n+1}, \bm{\varphi}_j)_{\Omega_f} + (\calmf_f\overline{\bm{u}}_h^n)_j, \\
        (\overline{\bm{f}}_2^{n+1})_j &= \frac{2\rho_p }{(\delta t)^2}(\calmf_p\overline{\bm{w}}_h^n)_j - \frac{\rho_p}{(\delta t)^2}(\calmf_p\overline{\bm{w}}_h^{n-1})_j + \frac{2\rho}{(\delta t)^2}(\calmf_p\overline{\bm{z}}_h^{n})_j - \frac{\rho}{(\delta t)^2}(\calmf_p\overline{\bm{z}}_h^{n-1})_j, \\
        (\overline{\bm{f}}_3^{n+1})_j &= 0, \\
        (\overline{\bm{f}}_4^{n+1})_j &= 0, \\
        (\overline{\bm{f}}_5^{n+1})_j &= (\calvf\overline{\bm{w}}_h^{n})_j,
    \end{aligned}
\end{equation}
and the vectors $\overline{\bm{u}}_h^{n+1}, \overline{\bm{w}}_h^{n+1}, \overline{\bm{z}}_h^{n+1}, \overline{\bm{p}}_h^{n+1}, \overline{\bm{g}}_h^{n+1}$ contain the FE coefficients for the variables $\bm{u}_h, w_h, z_h, p_h$, and $\bm{g}_h$, respectively.

Let $N_u$, $N_p$, $N_w$, $N_g$ be the total number of degrees of freedom for $\bm{u}_h$, $p_h$, $w_h$ (and $z_h$), and $\bm{g}_h$, respectively. Hence, the matrices in \eqref{eq:fe_matrices} have the following sizes:
\begin{equation*}
    \calmf_f, \calkf_f: N_u\times N_u, \quad \calmf_p, \calkf_p: N_w\times N_w, \quad \calpf: N_p\times N_u, \quad \caltf: N_g\times N_u, \quad \calvf: N_g\times N_w.
\end{equation*}
Following \cite{de_Castro2025}, for the partitioned method, we aim to express the vector $\overline{\bm{g}}_h^{n+1}$ as a function of the other variables. We begin by defining the following matrices:
\begin{equation}
    \calwf_f \coloneqq \calmf_f + \delta t\calkf_f, \quad \calaf_f \coloneqq \left[\begin{array}{c}
        \calpf \\
        \caltf
    \end{array}\right], \quad \calwf_p \coloneqq \left[\begin{array}{cc}
        \rho_p\calmf_p & \rho\calmf_p + (\delta t)^2\calkf_p \\
        -\delta t\calkf_p & \delta t\calmf_p 
    \end{array}\right], \quad \calaf_p \coloneqq \left[\begin{array}{cc}
        \bm{0} & \bm{0} \\
        \calvf & \bm{0} 
    \end{array}\right],
\end{equation}
and new variables:
\begin{equation}
    \overline{\bm{y}}_h^{n+1} \coloneqq \delta t\left[\begin{array}{c}
        \overline{\bm{p}}_h^{n+1} \\
        \overline{\bm{g}}_h^{n+1}
    \end{array}\right], \quad \overline{\bm{x}}_h^{n+1} \coloneqq \frac{1}{\delta t}\left[\begin{array}{c}
        \overline{\bm{w}}_h^{n+1} \\
        \overline{\bm{z}}_h^{n+1}
    \end{array}\right].
\end{equation}
Hence, we can rewrite the matrix equation \eqref{eq:fe_lin_system} as the following
\begin{align}
    &\calwf_f\overline{\bm{u}}_h^{n+1} - \calaf_f^T\overline{\bm{y}}_h^{n+1} = \overline{\bm{f}}_1^{n+1}, \label{eq:disc_fluid_eq} \\
    &\calwf_p\overline{\bm{x}}_h^{n+1} + \calaf_p^T\overline{\bm{y}}_h^{n+1} = \left[\begin{array}{c}
        \delta t\overline{\bm{f}}_2^{n+1} \\
        \overline{\bm{f}}_3^{n+1}
    \end{array}\right] \eqcolon \overline{\bm{d}}_1^{n+1}, \label{eq:disc_plate_eq} \\
    & \calaf_p\overline{\bm{x}}_h^{n+1} - \calaf_f \overline{\bm{u}}_h^{n+1} = \left[\begin{array}{c}
        \overline{\bm{f}}_4^{n+1} \\
      \frac{1}{\delta t}\overline{\bm{f}}_5^{n+1}
    \end{array}\right]\eqcolon \overline{\bm{d}}_2^{n+1}. \label{eq:disc_interf_eq}
\end{align}
From \eqref{eq:disc_fluid_eq}-\eqref{eq:disc_plate_eq}, we obtain
\begin{equation} \label{eq:uwz_solution}
    \overline{\bm{u}}_h^{n+1} = \calwf_f^{-1} \left(\overline{\bm{f}}_1^{n+1} + \calaf_f^T\overline{\bm{y}}_h^{n+1}\right), \quad \overline{\bm{x}}_h^{n+1} = \calwf_p^{-1}\left(\overline{\bm{d}}_1^{n+1} - \calaf_p^T\overline{\bm{y}}_h^{n+1}\right).
\end{equation}
Plugging these into \eqref{eq:disc_interf_eq} and rearranging the terms, we get
\begin{equation} \label{eq:schur_eq}
    \left(\calaf_f\calwf_f^{-1}\calaf_f^T + \calaf_p\calwf_p^{-1}\calaf_p^T\right)\overline{\bm{y}}_h^{n+1} = \calaf_p\calwf_p^{-1}\overline{\bm{d}}_1^{n+1} - \calaf_f\calwf_f^{-1}\overline{\bm{f}}_1^{n+1} - \overline{\bm{d}}_2^{n+1}.
\end{equation}
Equation \eqref{eq:schur_eq} is referred to as the Schur complement equation with the system matrix $\calsf = \calaf_f\calwf_f^{-1}\calaf_f^T + \calaf_p\calwf_p^{-1}\calaf_p^T$. After solving \eqref{eq:schur_eq} for $\overline{\bm{y}}_h^{n+1}$, 
other variables can be computed using \eqref{eq:uwz_solution}. Note that \eqref{eq:schur_eq} is uniquely solvable if $\calsf$ has a full rank. 
The conditioning of $\calsf$ and the associated matrices is analyzed in the next section. We summarize the partitioned algorithm in Alg.~\ref{alg:scur_dd_alg}.

\begin{algorithm}
\caption{Partitioned algorithm for the 3D fluid-2D plate system}
\label{algo1}
\label{alg:scur_dd_alg}
\begin{algorithmic}[1]
    \setlength{\itemsep}{1.3mm} 
    \Statex\textbf{Input:} $N\in \mathbb{N}$, and initial conditions $\bm{u}_0$, $w_0$, $w_{t0}$, $z_0$, and $z_{t0}$.
    \State Set $\overline{\bm{u}}_h^{0} = \bm{u}_0$, $\overline{\bm{w}}_h^{0} = w_0$, $\overline{\bm{w}}_h^{-1} = w^0 - \delta t \, w_{t0}$, $\overline{\bm{z}}_h^{0} = z_0$, and $\overline{\bm{z}}_h^{-1} = z^0 - \delta t \, z_{t0}$.
    \For{$n = 0, 1, \dots, N-1$}
        \State Compute $\overline{\bm{f}}_1^{n+1}$, $\overline{\bm{d}}_1^{n+1}$, and $\overline{\bm{d}}_2^{n+1}$ using \eqref{eq:rhs_vec}.
        \State Solve for $\overline{\bm{y}}_h^{n+1}$ in $\left(\calaf_f\calwf_f^{-1}\calaf_f^T + \calaf_p\calwf_p^{-1}\calaf_p^T\right)\overline{\bm{y}}_h^{n+1} = \calaf_p\calwf_p^{-1}\overline{\bm{d}}_1^{n+1} - \calaf_f\calwf_f^{-1}\overline{\bm{f}}_1^{n+1} - \overline{\bm{d}}_2^{n+1}$.
        \State Solve for $\overline{\bm{u}}_h^{n+1}$ in $\calwf_f\overline{\bm{u}}_h^{n+1} - \calaf_f^T\overline{\bm{y}}_h^{n+1} = \overline{\bm{f}}_1^{n+1}$.
        \State Solve for $\overline{\bm{x}}_h^{n+1}$ in $\calwf_p\overline{\bm{x}}_h^{n+1} + \calaf_p^T\overline{\bm{y}}_h^{n+1} = \overline{\bm{d}}_1^{n+1}$.
    \EndFor
\end{algorithmic}
\end{algorithm}

The Schur complement formulation \eqref{eq:schur_eq} is a partitioned approach in which the fluid and plate subproblems are solved 
independently, possibly in parallel. This approach offers several advantages over other methods. First, unlike other strongly coupled partitioned methods, this approach does not require subiterations between the two subproblems. Second, it enables the reuse of existing fluid and structural solvers without modifying their internal implementations. Since the Schur complement system is restricted to the pressure and LM unknowns, the computational effort associated with the coupled problem is significantly reduced compared with a fully monolithic approach.

A drawback of Schur complement methods is that the system matrix $\calsf$ is generally dense, even though the underlying matrices are sparse. Additionally, its conditioning may deteriorate as the mesh is refined, which can adversely affect the convergence of Krylov subspace methods. One possible approach is to explicitly assemble the Schur complement matrix once and solve the resulting system using a direct solver by just updating the right hand side every time step. However, because the Schur complement is dense, its explicit assembly can lead to substantial memory requirements and computational cost, particularly for large-scale problems. To avoid these issues, we employ a matrix-free implementation in which only matrix-vector products involving the matrices that constitute the system matrix $\calsf$ are evaluated. The resulting linear system is solved using the Generalized Minimal Residual (GMRES) method, while the independent fluid and plate subproblems \eqref{eq:disc_fluid_eq}-\eqref{eq:disc_plate_eq} are solved using the parallel sparse direct solver MUMPS \cite{MUMPS:1, MUMPS:2}. 

\section{Conditioning of the Schur complement matrix} \label{sec:cond_num_schurc}

The objective of this section is to quantify the conditioning of the Schur complement matrix and identify the dominant source of possible ill-conditioning. To this end, we first rewrite the plate block in a more convenient form, allowing the Schur complement to be expressed as the product $\widetilde{\calbf}\widetilde{\calwf}^{-1}\widetilde{\calbf}^T$.

To simplify the analysis of the Schur complement matrix, we first rewrite the plate block $\calwf_p$. 
For the sake of simplicity, we define the notations for the blocks as the following: $A = \rho_p\calmf_p$, $B = \rho\calmf_p + (\delta t)^2\calkf_p$, $C = -\delta t\calkf_p$, and $D = \delta t\calmf_p$. 
Note that since the mass $\calmf_p$ and stiffness $\calkf_p$ matrices are invertible, we have \cite{Bernstein2005}
\begin{equation}
    \calwf_p^{-1} = \left[\begin{array}{cc}
        (A - BD^{-1}C)^{-1} & -(A - BD^{-1}C)^{-1}BD^{-1} \\
        -(D - CA^{-1}B)^{-1}CA^{-1} & (D-CA^{-1}B)^{-1}
    \end{array}\right] \eqcolon \left[\begin{array}{cc}
        E & F \\
        G & H
    \end{array}\right].
\end{equation}
This means that we can simplify $\calaf_p\calwf_p^{-1}\calaf_p^T$ as 
\begin{align*}
    \calaf_p\calwf_p^{-1}\calaf_p^T &= \left[\begin{array}{cc}
        \bm{0}_{N_p\times N_w} & \bm{0}_{N_p\times N_w} \\
         \calvf & \bm{0}_{N_g\times N_w}
    \end{array}\right]\left[\begin{array}{cc}
        E & F \\
        G & H 
    \end{array}\right]\left[\begin{array}{cc}
        \bm{0}_{N_w\times N_p} & \calvf^T \\
         \bm{0}_{N_w\times N_p} & \bm{0}_{N_w\times N_g}
    \end{array}\right]\\
    &= \left[\begin{array}{cc}
        \bm{0}_{N_p\times N_p} & \bm{0}_{N_p\times N_g} \\
        \bm{0}_{N_g\times N_p} & \calvf E\calvf^T
    \end{array}\right].
\end{align*}
For consistency of notation, we define 
\begin{equation} \label{eq:tildeWp_def}
    \widetilde{\calwf}_p \coloneqq E^{-1} = A - BD^{-1}C = \rho_p\calmf_p + \rho\calkf_p + (\delta t)^2\calkf_p\calmf_p^{-1}\calkf_p.
\end{equation}
Note that since the mass and stiffness matrices are symmetric and positive definite (spd), $\widetilde{\calwf}_p$ is also spd. Hence, we introduce the following matrices:
\begin{equation} \label{eq:tilde_matrices}
    \widetilde{\calaf}_p \coloneqq \left[\begin{array}{c}
        \bm{0}_{N_p\times N_w} \\
        \calvf 
    \end{array}\right], \quad
    \widetilde{\calwf} \coloneqq \left[\begin{array}{cc}
        \calwf_f & \bm{0} \\
        \bm{0} & \widetilde{\calwf}_p
    \end{array}\right], \quad \widetilde{\calbf} \coloneqq \left[\calaf_f \quad \widetilde{\calaf}_p\right] = \left[\begin{array}{cc}
        \calpf & \bm{0} \\
        \caltf & \calvf
    \end{array}\right].
\end{equation}
It is easy to see that
\begin{equation} \label{eq:S-definition1}
    \calsf = \calaf_f\calwf_f^{-1}\calaf_f^T + \widetilde{\calaf}_p\widetilde{\calwf}_p^{-1}\widetilde{\calaf}_p^T \eqcolon \widetilde{\calbf}\widetilde{\calwf}^{-1}\widetilde{\calbf}^T.
\end{equation}
Following the similar steps in \cite[Sec. 5]{de_Castro2025}, we have
\begin{equation}\label{eq:cond_number_reform}
    \kappa(\calsf) \leq (\kappa(\widetilde{\calbf}))^2\kappa(\widetilde{\calwf}).
\end{equation}

We repeatedly use the following standard finite element estimates, see e.g., \cite{Quarteroni_Valli1994}:
\begin{enumerate}
    \item For any finite element function $v_h = \sum_i v_i\phi_i$, there exist positive constants $C_\ell$, $C_u$  such that 
    \begin{equation}
    \label{eq:vec-func_bound}
        C_\ell h^d \abs{\overline{\bm{v}}}^2 \leq \norm{v_h}_{\Omega}^2 \leq C_u h^d \abs{\overline{\bm{v}}}^2;
    \end{equation}
    where $\overline{\bm{v}}$ is the vector containing the coefficients $v_i$ and $\abs{\overline{\bm{v}}}$ is the Euclidean norm of $\overline{\bm{v}}$.
    \item Stiffness matrix eigenvalues: there exist positive constants $c_i^\ell$, $i=1,2,3$  such that
    \begin{equation} \label{eq:stiff_eig_bound}
        c_1^\ell\nu h_\ell^d \leq \lambda_i(\calkf_\ell) \leq c_2^\ell \nu h_\ell^d(1+c^\ell_3 h_\ell^{-2}),
    \end{equation}
    where $d = 3$ for the fluid, $d = 2$ for the plate, $\nu = 2\nu_f$ if $\ell = f$, and $\nu = 1$ if $\ell = p$.
    \item Mass matrix eigenvalues: there exist positive constants $c_i^\ell$, $i=1,2$ such that    \begin{equation}\label{eq:mass_eig_bound}
        \bar{c}_1^\ell\rho_\ell h_\ell^d \leq \lambda_i(\calmf_\ell) \leq \bar{c}_2^\ell \rho_\ell h_\ell^d,
    \end{equation}
    where $d = 3$ for the fluid and $d = 2$ for the plate, and $\rho_\ell = 1$ if $\ell=p$. 
\end{enumerate}

\begin{lemma} \label{lem:wtilde_cond}
    We have the following estimate for the condition number $\kappa(\widetilde{\calwf})$ of $\widetilde{\calwf}$:
    \begin{equation} \label{eq:W_cond}
        \kappa(\widetilde{\calwf}) \leq \frac{\max\left\{C_f(h_f^3 + \delta th_f^3 + \delta t h_f), C_p(h_p^2 + 1 + (\delta t)^2(h_p^2 + 1 + h_p^{-2}))\right\}}{\min\left\{c_f h_f^3(1 + \delta t), c_p(1 + (\delta t)^2)h_p^2\right\}},
    \end{equation}
    for some constants $C_{\ell}, \,c_{\ell} >0$, $\ell=f,p$,  independent of the mesh sizes, $h_f$ and $h_p$, and the time step $\delta t$.
\end{lemma}

\begin{proof}
 From the definition of $\calwf_f$ and the classical results 
  \eqref{eq:stiff_eig_bound} and \eqref{eq:mass_eig_bound}, we have the following estimate for the eigenvalues of $\calwf_f$:
    \begin{equation} \label{eq:Wf_lambda}
        \left(\bar{c}_1^f\rho_f + 2\delta t c_1^f \nu_f \right)h_f^3 \leq \lambda_i(\calwf_f) \leq \left(\bar{c}_2^f\rho_f + 2\delta t c_2^f \nu_f\left(1 + c_3^f h_f^{-2}\right)\right) h_f^3,
    \end{equation}
    which is the same bound found in \cite{de_Castro2025}.

    We denote the maximum eigenvalue of a matrix $\calaf\in \mathbb{R}^{N\times N}$ by $\lambda_\text{max}(\calaf)$ and the minimum eigenvalue by $\lambda_\text{min}(\calaf)$. Hence, for any eigenvalue $\lambda_i(\calaf)$ of $\calaf$,
    \begin{equation*}
        \lambda_\text{min}(\calaf) \leq \lambda_i(\calaf) \leq \lambda_\text{max}(\calaf).
    \end{equation*}
    Note that from the Rayleigh-Ritz theorem \cite[Theorem 4.2.2]{Horn_Johnson1085}, we have the following relation for any symmetric matrix $\calaf$:
    \begin{align}
        \lambda_\text{min}(\calaf) \leq &\frac{\overline{\bm{x}}^T\calaf\overline{\bm{x}}}{\abs{\overline{\bm{x}}}^2} \leq \lambda_\text{max}(\calaf) \label{eq:RRT-1} \\
        \lambda_\text{min}(\calaf) = \min_{\overline{\bm{x}}\in \mathbb{R}^N} \frac{\overline{\bm{x}}^T\calaf\overline{\bm{x}}}{\abs{\overline{\bm{x}}}^2}, &\quad \lambda_\text{max}(\calaf) = \max_{\overline{\bm{x}}\in \mathbb{R}^N} \frac{\overline{\bm{x}}^T\calaf\overline{\bm{x}}}{\abs{\overline{\bm{x}}}^2},\label{eq:RRT-2}
    \end{align}
    where $\abs{\overline{\bm{x}}}^2 = \overline{\bm{x}}^T\overline{\bm{x}}$. Note that since $\calmf_p$, and $\calkf_p$ are spd, $\calkf_p\calmf_p^{-1}\calkf_p$ is also spd. If $\overline{\bm{y}} = \calkf_p\overline{\bm{x}}$, then by \eqref{eq:RRT-1}-\eqref{eq:RRT-2}, we have
    \begin{equation}
        \overline{\bm{y}}^T\calmf_p^{-1}\overline{\bm{y}} \geq \lambda_\text{min}(\calmf_p^{-1})\abs{\overline{\bm{y}}}^2 = \frac{1}{\lambda_\text{max}(\calmf_p)}\abs{\overline{\bm{y}}}^2 = \frac{1}{\lambda_\text{max}(\calmf_p)}\abs{\calkf_p\overline{\bm{x}}}^2, \label{eq:RRT-3}
    \end{equation}
    for any $\overline{\bm{x}}\in\mathbb{R}^{N_w}$. By the same arguments and the symmetry of $\calkf_p$, we have
    \begin{equation}
        \abs{\calkf_p\overline{\bm{x}}}^2 = \overline{\bm{x}}^T\calkf_p^T\calkf_p\overline{\bm{x}} = \overline{\bm{x}}^T\calkf_p^2\overline{\bm{x}} \geq \lambda_\text{min}(\calkf_p)^2\abs{\overline{\bm{x}}}^2. \label{eq:RRT-4}
    \end{equation}
    Hence, by \eqref{eq:RRT-2}-\eqref{eq:RRT-4}, we obtain
    \begin{equation*}
        \lambda_\text{min}(\calkf_p\calmf_p^{-1}\calkf_p) \geq \frac{\lambda_\text{min}(\calkf_p)^2}{\lambda_\text{max}(\calmf_p)}.
    \end{equation*} 
    Similarly, we have
    \begin{equation*} 
        \lambda_\text{max}(\calkf_p\calmf_p^{-1}\calkf_p) \leq \frac{\lambda_\text{max}(\calkf_p)^2}{\lambda_\text{min}(\calmf_p)}.
    \end{equation*}
    This means that for any eigenvalue $\lambda_i(\calkf_p\calmf_p^{-1}\calkf_p)$ of the matrix $\calkf_p\calmf_p^{-1}\calkf_p$, we have the following estimate
    \begin{equation*}
        \frac{\lambda_\text{min}(\calkf_p)^2}{\lambda_\text{max}(\calmf_p)} \leq \lambda_i(\calkf_p\calmf_p^{-1}\calkf_p) \leq \frac{\lambda_\text{max}(\calkf_p)^2}{\lambda_\text{min}(\calmf_p)}.
    \end{equation*}
    Hence, there exist constants $C_1, C_2>0$ such that
    \begin{equation} \label{eq:temp_eig_bound}
       C_1 h_p^2 \leq \lambda_i(\calkf_p\calmf_p^{-1}\calkf_p) \leq C_2(h_p^2 + 1 + h_p^{-2}).
    \end{equation}
    Using  \eqref{eq:tildeWp_def}, \eqref{eq:stiff_eig_bound} \eqref{eq:mass_eig_bound} and \eqref{eq:temp_eig_bound},
    we obtain the estimate 
    \begin{equation*}
        \bar{c}_1^p\rho_p h_p^2 + c_1^p\rho h_p^2 + C_1(\delta t)^2h_p^2 \leq \lambda_i(\widetilde{\calwf}_p) \leq \bar{c}_2^p\rho_ph_p^2 + c_2^p\rho h_p^2(1 + c_3^ph_p^{-2}) + C_2(\delta t)^2(h_p^2 + 1 + h_p^{-2}),
    \end{equation*}
    which means there exist $C_p, c_p > 0$ independent of $h_p$ and $\delta t$ such that
    \begin{equation} \label{eq:Wp_lambda}
        \begin{aligned}
            \lambda_\text{min}(\widetilde{\calwf}_p) \geq c_p(1 + (\delta t)^2)h_p^2,\quad \lambda_\text{max}(\widetilde{\calwf}_p) \leq C_p (h_p^2 + 1 + (\delta t)^2(h_p^2 + 1 + h_p^{-2})).
        \end{aligned}
    \end{equation}
    Thus, using \eqref{eq:Wf_lambda} and \eqref{eq:Wp_lambda}, we have the following estimate for $\kappa\left(\widetilde{\calwf}\right)$:
    \begin{equation*}
        \begin{aligned}
            \kappa(\widetilde{\calwf}) &= \frac{\lambda_\text{max}(\widetilde{\calwf})}{\lambda_\text{min}(\widetilde{\calwf})} \leq \frac{\max\left\{C_f(h_f^3 + \delta th_f^3 + \delta t h_f), C_p(h_p^2 + 1 + (\delta t)^2(h_p^2 + 1 + h_p^{-2}))\right\}}{\min\left\{c_f h_f^3(1 + \delta t), c_p (1 + (\delta t)^2)h_p^2\right\}}.
        \end{aligned}
    \end{equation*}
    for some positive constants $c_f, C_f$ which only depend on physical parameters.
\end{proof}

Next, we provide an estimate for the condition number of $\widetilde{\calbf}$. Note that the condition number of a rectangular matrix is given by
\begin{equation*}
    \kappa(\widetilde{\calbf}) = \frac{\sigma_\text{max}(\widetilde{\calbf})}{\sigma_\text{min}(\widetilde{\calbf})},
\end{equation*}
where $\sigma_\text{max}(\widetilde{\calbf})$ and $\sigma_\text{min}(\widetilde{\calbf})$ are the largest and smallest singular values of $\widetilde{\calbf}$, respectively.

\begin{lemma} \label{lem:btilde_cond}
    We have the following estimate for the condition number $\kappa(\widetilde{\calbf})$ of $\widetilde{\calbf}$:
    \begin{equation} \label{eq:B_cond}
        \kappa(\widetilde{\calbf}) \leq \bar{C}\left(\frac{(h_f + 1)(h_f^3 + h_p^2)}{\min\{h_f^3, h_p^2\}\min\{h_f^3, h_p^3\}}\right)^{1/2}
    \end{equation}
    for some positive constant $\bar{C}$ independent of the mesh sizes $h_f, h_p$, and time step $\delta t$.
\end{lemma}

\begin{proof}
    Define the spaces $X = \bm{U}_h \times W_h$ and $Y = Q_h \times \bm{G}_h$ with norms defined by
    \begin{equation*}
        \norm{(\bm{v}_h, \eta_h)}_{X}^2 \coloneqq \norm{\bm{v}_h}_{1, \Omega_f}^2 + \norm{\eta_h}_{1, \Omega_p}^2 \quad \text{and} \quad \norm{(q_h, \bm{\lambda}_h)}_Y^2 \coloneqq \norm{q_h}_{\Omega_f}^2 + \norm{\bm{\lambda}_h}_{-1/2, \Omega_p}^2,
    \end{equation*}
    respectively.  Define the orthogonal matrix
    \begin{equation*}
       \caldf_R \coloneqq \left[\begin{array}{cc}
           - \calif_{N_u} & \bm{0} \\
            \bm{0} & \calif_{N_w}
        \end{array}\right]\in\mathbb{R}^{(N_u+N_w)\times (N_u+N_w)},
    \end{equation*}
    where $\calif_N$ is the identity matrix in $\mathbb{R}^{N\times N}$, and 
    \begin{equation*}
        \calbf_2 \coloneqq  \widetilde{\calbf}\caldf_R = \left[\begin{array}{cc}
            -\calpf & \bm{0} \\
            -\caltf & \calvf
        \end{array}\right].
    \end{equation*}
    We first note that if $\overline{\bm{x}} = (\overline{\bm{u}}_h, \overline{\bm{w}}_h)^T$ and $\overline{\bm{y}} = (\overline{\bm{p}}_h, \overline{\bm{g}}_h)^T$ are the FE coefficients corresponding to $(\bm{u}_h, w_h)\in \bm{U}_h\times W_h$ and $(p_h, \bm{g}_h)\in Q_h \times \bm{G}_h$, respectively, then 
    \begin{equation*}
        \overline{\bm{x}}^T \calbf_2^T\overline{\bm{y}} = b\left((\bm{u}_h, w_h), (p_h, \bm{g}_h)\right).
    \end{equation*}
   Since $\caldf_R$ is an orthogonal matrix, $\sigma_i(\calbf_2) = \sigma_i(\widetilde{\calbf})$, and thus, $\kappa(\calbf_2) = \kappa(\widetilde{\calbf})$. Hence, it suffices to compute for the condition number of $\calbf_2$.
    We first show the boundedness of the bilinear form $b(\cdot, \cdot)$:
    \begin{equation*}
        \begin{aligned}
            b\left((\bm{v}_h, \eta_h), (q_h, \bm{\lambda}_h)\right) &= -(\nabla\cdot\bm{v}_h, q_h)_{\Omega_f} - \dual{\bm{v}_h, \bm{\lambda}_h}_{\Omega_p} + \dual{\eta_h \bm{n}, \bm{\lambda}_h}_{\Omega_p} \\
            &\leq \norm{\nabla\cdot\bm{v}_h}_{\Omega_f}\norm{q_h}_{\Omega_f} + \norm{\bm{v}_h}_{1/2, \Omega_p}\norm{\bm{\lambda}_h}_{-1/2, \Omega_p}\\
            & \qquad+ \norm{\eta_h}_{1/2, \Omega_p}\norm{\bm{\lambda}_h}_{-1/2, \Omega_p} \\
            &\leq C\left(\norm{\bm{v}_h}_{1, \Omega_f} + \norm{\eta_h}_{1, \Omega_p}\right)\left(\norm{q_h}_{\Omega_f} + \norm{\bm{\lambda}_h}_{-1/2, \Omega_p}\right) \\
            &\leq 2C \left(\norm{\bm{v}_h}_{1, \Omega_f}^2 + \norm{\eta_h}_{1, \Omega_p}^2\right)^{1/2}\left(\norm{q_h}_{\Omega_f}^2 + \norm{\bm{\lambda}_h}_{-1/2, \Omega_p}^2\right)^{1/2}\\
            &= C^*\norm{(\bm{v}_h, \eta_h)}_{X}\norm{(q_h, \bm{\lambda}_h)}_Y.
        \end{aligned}
    \end{equation*}
    Note that the smallest and largest singular values of $\calbf_2$ are given by
    \begin{equation*}
        \sigma_\text{min}(\calbf_2^T) = \min_{\overline{\bm{y}}\in \mathbb{R}^{N_p + N_g}}\max_{\overline{\bm{x}}\in \mathbb{R}^{N_u + N_w}}\frac{\overline{\bm{x}}^T \calbf_2^T\overline{\bm{y}}}{\abs{\overline{\bm{x}}}\abs{\overline{\bm{y}}}} \quad \text{and} \quad \sigma_\text{max}(\calbf_2^T) = \max_{\overline{\bm{y}}\in \mathbb{R}^{N_p + N_g}}\max_{\overline{\bm{x}}\in \mathbb{R}^{N_u + N_w}}\frac{\overline{\bm{x}}^T \calbf_2^T\overline{\bm{y}}}{\abs{\overline{\bm{x}}}\abs{\overline{\bm{y}}}}.
    \end{equation*}
    Let $\bm{x}_h = (\bm{u}_h, w_h)\in X$ and $\bm{y}_h = (p_h, \bm{g}_h)\in Y$. Then, using Poincar\'{e}-Friedrich inequality, the inverse inequality and \eqref{eq:vec-func_bound}, 
    \begin{equation*}
        \begin{aligned}
            \frac{\norm{\bm{x}_h}_{X}^2}{\abs{\overline{\bm{x}}}^2} &= \frac{\norm{\bm{u}_h}_{1, \Omega_f}^2 + \norm{w_h}_{1,\Omega_p}^2}{\abs{\overline{\bm{u}}_h}^2 + \abs{\overline{\bm{w}}}_h^2} \leq C_P\left(\frac{\norm{\nabla\bm{u}_h}_{\Omega_f}^2 + \norm{\nabla w_h}_{\Omega_p}^2}{\abs{\overline{\bm{u}}_h}^2 + \abs{\overline{\bm{w}}}_h^2}\right) \\
            & \leq C_P\left(\frac{Ch_f^{-2}\norm{\bm{u}_h}_{\Omega_f}^2}{\abs{\overline{\bm{u}}_h}^2} + \frac{Ch_p^{-2}\norm{w_h}_{\Omega_p}^2}{\abs{\overline{\bm{w}}}_h^2}\right) \\
            & \leq C_P\left(CC_uh_f + CC_u\right) \\
            & \leq C(h_f + 1).
        \end{aligned}
    \end{equation*}
    Similarly, we obtain
    \begin{equation*}
        \begin{aligned}
            \frac{\norm{\bm{y}_h}_{Y}^2}{\abs{\overline{\bm{y}}}^2} = \frac{\norm{p_h}_{\Omega_f}^2 + \norm{\bm{g}_h}_{-1/2,\Omega_p}^2}{\abs{\overline{\bm{p}}_h}^2 + \abs{\overline{\bm{g}}_h}^2} \leq \frac{\norm{p_h}_{\Omega_f}^2}{\abs{\overline{\bm{p}}_h}^2} + \frac{\norm{\bm{g}_h}_{-1/2,\Omega_p}^2}{\abs{\overline{\bm{g}}_h}^2} \leq \frac{\norm{p_h}_{\Omega_f}^2}{\abs{\overline{\bm{p}}_h}^2} + \frac{\norm{\bm{g}_h}_{\Omega_p}^2}{\abs{\overline{\bm{g}}_h}^2} \leq C_u(h_f^3 + h_p^2).
        \end{aligned}
    \end{equation*}
    Hence, combining the two bounds and the continuity of $b(\cdot, \cdot)$,  we get an upper bound for $\sigma_\text{max}(\calbf_2)$:
    \begin{equation*} 
        \frac{\overline{\bm{x}}^T \calbf_2^T\overline{\bm{y}}}{\abs{\overline{\bm{x}}}\abs{\overline{\bm{y}}}} = \frac{b\left((\bm{u}_h, w_h), (p_h, \bm{g}_h)\right)}{\abs{\overline{\bm{x}}}\abs{\overline{\bm{y}}}} \leq \frac{C^*\norm{\bm{x}_h}_{X}\norm{\bm{y}_h}_Y}{\abs{\overline{\bm{x}}}\abs{\overline{\bm{y}}}} \leq C(h_f + 1)^{1/2}(h_f^3 + h_p^2)^{1/2},
    \end{equation*}
    and
    \begin{equation*}
        \sigma_\text{max}(\calbf_2^T) \leq C(h_f + 1)^{1/2}(h_f^3 + h_p^2)^{1/2},
    \end{equation*}
    for some $C>0$ independent of the mesh sizes, $h_f$ and $h_p$, and the time step $\delta t$. Now, we observe that using the inverse inequality \cite{de_Castro2025, Temam1984},  $\norm{\bm{\lambda}_h}_{\Omega_p} \leq Ch_p^{-1/2}\norm{\bm{\lambda}_h}_{-1/2, \Omega_p}$, 
  and \eqref{eq:vec-func_bound}, 
    we obtain
    \begin{equation*}
        \begin{aligned}
            \frac{\norm{\bm{x}_h}_{X}^2\norm{\bm{y}_h}_Y^2}{\abs{\overline{\bm{x}}}^2\abs{\overline{\bm{y}}}^2} &= \left(\frac{(\norm{\bm{u}_h}_{1, \Omega_f}^2 + \norm{w_h}_{1, \Omega_p}^2}{\abs{\overline{\bm{u}}_h}^2 + \abs{\overline{\bm{w}}_h}^2}\right)\left(\frac{\norm{p_h}_{\Omega_f}^2 + \norm{\bm{\lambda}_h}_{-1/2, \Omega_p}^2}{\abs{\overline{\bm{p}}_h}^2 + \abs{\overline{\bm{g}}_h}^2}\right) \\
            &\geq \left(\frac{\norm{\bm{u}_h}_{\Omega_f}^2 + \norm{w_h}_{\Omega_p}^2}{\abs{\overline{\bm{u}}_h}^2 + \abs{\overline{\bm{w}}_h}^2}\right)\left(\frac{\norm{p_h}_{\Omega_f}^2 + C^{-1}h_p\norm{\bm{\lambda}_h}_{\Omega_p}^2}{\abs{\overline{\bm{p}}_h}^2 + \abs{\overline{\bm{g}}_h}^2}\right) \\
            &\geq \left(\frac{C_\ell h_f^3\abs{\overline{\bm{u}}_h}^2 + C_\ell h_p^2\abs{\overline{\bm{w}}_h}^2}{\abs{\overline{\bm{u}}_h}^2 + \abs{\overline{\bm{w}}_h}^2}\right)\left(\frac{C_\ell h_f^3\abs{\overline{\bm{p}}_h}^2 + C^{-1}C_\ell h_p^3\abs{\overline{\bm{g}}_h}^2}{\abs{\overline{\bm{p}}_h}^2 + \abs{\overline{\bm{g}}_h}^2}\right) \\
            & \geq C\min\{h_f^3, h_p^2\}\min\{h_f^3, h_p^3\}.
        \end{aligned}
    \end{equation*}
    Now, using \eqref{eq:full_disc_infsup}, we get
    \begin{equation*}
        \begin{aligned}
            \sigma_\text{min}(\calbf_2^T) &= \min_{\overline{\bm{y}} \in \mathbb{R}^{N_p+N_g}}\max_{\overline{\bm{x}}\in\mathbb{R}^{N_u+N_w}}\frac{\overline{\bm{x}}^T \calbf_2^T\overline{\bm{y}}}{\abs{\overline{\bm{x}}}\abs{\overline{\bm{y}}}} \\
            &= \inf_{\bm{y}_h\in Y}\sup_{\bm{x}_h\in X}\frac{b\left((\bm{u}_h, w_h), (p_h, \bm{g}_h)\right)}{\abs{\overline{\bm{x}}}\abs{\overline{\bm{y}}}} \\ 
            &= \inf_{\bm{y}_h\in Y}\sup_{\bm{x}_h\in X}\frac{b\left(\bm{x}_h, \bm{y}_h\right)}{\norm{\bm{x}_h}_{X}\norm{\bm{y}_h}_Y} \cdot \frac{\norm{\bm{x}_h}_{X}\norm{\bm{y}_h}_Y}{\abs{\overline{\bm{x}}}\abs{\overline{\bm{y}}}} \\
            & \geq \tilde{C}\beta\left(\min\{h_f^3, h_p^2\}\min\{h_f^3, h_p^3\}\right)^{1/2},
        \end{aligned}
    \end{equation*}
    and thus, 
    \begin{equation*}
        \sigma_\text{min}(\calbf_2^T) \geq \tilde{C}\beta\left(\min\{h_f^3, h_p^2\}\min\{h_f^3, h_p^3\}\right)^{1/2}.
    \end{equation*}
    From the definition of the condition number we have
    \begin{equation*}
        \kappa(\widetilde{\calbf}) = \kappa(\calbf_2) \leq \bar{C}\left(\frac{(h_f + 1)(h_f^3 + h_p^2)}{\min\{h_f^3, h_p^2\}\min\{h_f^3, h_p^3\}}\right)^{1/2}.
    \end{equation*}
\end{proof}

Hence, from Lemmas~\ref{lem:wtilde_cond}-\ref{lem:btilde_cond}, we have the following estimate for the condition number of the Schur complement system $S$.

\begin{theorem} \label{thm:S_cond}
    There exists a constant $C$, independent of the mesh sizes, $h_f$ and $h_p$, and time step $\delta t$ such that
    \begin{equation}\label{eq:S_cond}
        \kappa(\calsf) \leq C \frac{\max\left\{C_u(h_f^3 + \delta th_f^3 + \delta t h_f), c_2(h_p^2 + 1 + (\delta t)^2(h_p^2 + 1 + h_p^{-2}))\right\}(h_f + 1)(h_f^3 + h_p^2)}{\min\left\{c_uh_f^3(1 + \delta t), c_1(1 + (\delta t)^2)h_p^2\right\}\min\{h_f^3, h_p^2\}\min\{h_f^3, h_p^3\}}.
    \end{equation}
\end{theorem}

The estimate in Theorem~\ref{thm:S_cond} also reveals that the dominant contribution to the condition number arises from the factor $\kappa(\widetilde{\calbf})^2$. 
If $h= h_f = h_p $, for example,  $\kappa(\widetilde{\calbf})=O(h^{-2})$, and this term contributes $O(h^{-4})$ to the growth of $\kappa(\calsf)$, while the remaining factor $\kappa(\widetilde{\calwf})$ contributes only $O(h^{-3})$. Note that although $\delta t$ appears in the estimate in \eqref{eq:S_cond}, $\kappa(\calsf)$ does not depend on $\delta t$. Thus, the ill-conditioning of the Schur complement is driven primarily by the interface and divergence matrices. This observation motivates considering a preconditioner of the form $\widetilde{\calbf}\widetilde{\calbf}^{T}$, which is designed to eliminate the dependence on $\kappa(\widetilde{\calbf})^2$. The next theorem shows that this choice yields a significantly better-conditioned Schur complement system.

\begin{theorem} \label{thm:precondS_cond}
Given that
\begin{equation*}
    \calsf = \widetilde{\calbf}\,\widetilde{\calwf}^{-1}\,
    \widetilde{\calbf}^{\,T},
\end{equation*}
where $\widetilde{\calwf}$ and
$\widetilde{\calbf}$ are defined as in \eqref{eq:tilde_matrices}, define the preconditioner
\begin{equation}\label{eq:bbt_precond}
    \calmf = \widetilde{\calbf}\,\widetilde{\calbf}^{\,T}.
\end{equation}
Then
\begin{equation*}
    \kappa\left(\calmf^{-1}\calsf\right)
    \leq \kappa(\widetilde{\calwf}).
\end{equation*}
\end{theorem}
\begin{proof}
    Note that since $\widetilde{\calbf}$ is full row rank, $\calmf$ is spd. This means that $\calmf^{1/2}$ and $\calmf^{-1/2}$ are well-defined through the diagonalization of $\calmf$. Note that they are also symmetric.  It is also easy to see that $\calmf^{-1}\calsf$ and $\calmf^{-1/2}\calsf\calmf^{-1/2}$ are similar matrices. Thus, they have the same eigenvalues, i.e., $\lambda_i(\calmf^{-1}\calsf) = \lambda_i(\calmf^{-1/2}\calsf\calmf^{-1/2})$. \\

    Let $\overline{\bm{z}} \in \mathbb{R}^{N_p + N_g}$ be a nonzero vector. Define $\overline{\bm{y}} = \calmf^{-1/2}\overline{\bm{z}}$ so that
    \begin{equation*}
        \overline{\bm{z}}^T\calmf^{-1/2}\calsf\calmf^{-1/2}\overline{\bm{z}} = \overline{\bm{y}}^T\calsf\overline{\bm{y}}
    \end{equation*}
    and
    \begin{equation*}
        \overline{\bm{z}}^T\overline{\bm{z}} = \overline{\bm{y}}^T\calmf\overline{\bm{y}}.
    \end{equation*}
    Hence, we obtain that
    \begin{equation*}
\frac{\overline{\bm{z}}^T\calmf^{-1/2}\calsf\calmf^{-1/2}\overline{\bm{z}}}{\overline{\bm{z}}^T\overline{\bm{z}}} = \frac{\overline{\bm{y}}^T\calsf\overline{\bm{y}}}{\overline{\bm{y}}^T\calmf\overline{\bm{y}}}.
    \end{equation*}
    Moreover, we can define the vector $\overline{\bm{r}} = \widetilde{\calbf}^T\overline{\bm{y}}$ so that
    \begin{equation*}
        \overline{\bm{y}}^T\calsf\overline{\bm{y}} = \overline{\bm{y}}^T\widetilde{\calbf}\widetilde{\calwf}^{-1}\widetilde{\calbf}^T\overline{\bm{y}} = \overline{\bm{r}}^T\widetilde{\calwf}^{-1}\overline{\bm{r}}
    \end{equation*}
    and 
    \begin{equation*}
        \overline{\bm{y}}^T\calmf\overline{\bm{y}} = \overline{\bm{y}}^T\widetilde{\calbf}\widetilde{\calbf}^T\overline{\bm{y}} = \overline{\bm{r}}^T\overline{\bm{r}}.
    \end{equation*}
    Therefore,
    \begin{equation*}
        \frac{\overline{\bm{z}}^T\calmf^{-1/2}\calsf\calmf^{-1/2}\overline{\bm{z}}}{\overline{\bm{z}}^T\overline{\bm{z}}} = \frac{\overline{\bm{r}}^T\widetilde{\calwf}^{-1}\overline{\bm{r}}}{\overline{\bm{r}}^T\overline{\bm{r}}}.
    \end{equation*}
    Since $\overline{\bm{z}}\neq 0$ and $\calmf^{-1/2}$ is nonsingular, we have $\overline{\bm{y}}\neq 0$. Moreover, since $\widetilde{\calbf}$ has full row rank, $\widetilde{\calbf}^T$ has trivial nullspace, and therefore $\overline{\bm{r}}=\widetilde{\calbf}^T\overline{\bm{y}}\neq 0$. Hence, by the Rayleigh--Ritz theorem,
    \begin{equation*}
        \lambda_{\min}(\widetilde{\calwf}^{-1}) \leq \frac{\overline{\bm{r}}^T\widetilde{\calwf}^{-1}\overline{\bm{r}}}{\overline{\bm{r}}^T\overline{\bm{r}}} \leq \lambda_{\max}(\widetilde{\calwf}^{-1}).
    \end{equation*}
    Since the above Rayleigh quotient bounds hold for every nonzero $\overline{\bm{z}}$, the eigenvalues of  $\calmf^{-1/2}\calsf\calmf^{-1/2}$ lie in the same interval. Therefore,
    \begin{equation*}
        \lambda_{\min}(\widetilde{\calwf}^{-1}) \leq \lambda_i(\calmf^{-1}\calsf) \leq \lambda_{\max}(\widetilde{\calwf}^{-1})
    \end{equation*}
    and
    \begin{equation*}
        \kappa(\calmf^{-1}\calsf) \leq \kappa(\widetilde{\calwf}^{-1}) = \kappa(\widetilde{\calwf}).
    \end{equation*}

\end{proof}

\begin{remark}
    The estimate in Theorem~\ref{thm:S_cond} shows that the condition number of the Schur complement matrix $\calsf$ increases as the mesh is refined and is also sensitive to the combinations between the fluid mesh size $h_f$ and the plate mesh size $h_p$. In particular, if $h_f = h_p = h$, the estimate becomes $\kappa(S) = O(h^{-7})$, since $\kappa(\widetilde{\calwf}) = O(h^{-3})$ and $ \kappa(\widetilde{\calbf}) = O(h^{-2})$. However, Theorem~\ref{thm:precondS_cond} indicates that the preconditioned Schur complement matrix satisfies $\kappa\left(\calmf^{-1}\calsf\right)
    \leq \kappa(\widetilde{\calwf})=O(h^{-3})$. Therefore, the proposed preconditioner significantly improves the conditioning of the Schur complement matrix, which is also verified by the numerical tests presented in the next section. 
\end{remark}

\section{Numerical Results} \label{sec:num_res}

In this section, we evaluate the accuracy and efficiency of the proposed partitioned algorithm through a series of numerical experiments. We first verify the spatial and temporal convergence rates using manufactured solutions. We then investigate the performance of the Schur complement solver by comparing the unpreconditioned system, the proposed preconditioner, and the block Jacobi preconditioner
\begin{equation}
    \calmf \coloneqq
    \begin{bmatrix}
        \calpf\calwf_f^{-1}\calpf^T & \bm{0} \\
        \bm{0} &
        \caltf\calwf_f^{-1}\caltf^T
        + \calvf\calwf_p^{-1}\calvf^T
    \end{bmatrix},
\end{equation}
which is obtained by retaining only the diagonal blocks of the Schur complement matrix
\begin{equation}
    \calsf =
    \begin{bmatrix}
        \calpf\calwf_f^{-1}\calpf^T &
        \calpf\calwf_f^{-1}\caltf^T \\
        \caltf\calwf_f^{-1}\calpf^T &
        \caltf\calwf_f^{-1}\caltf^T
        + \calvf\widetilde{\calwf_p}^{-1}\calvf^T
    \end{bmatrix}.
\end{equation}
The effectiveness of these approaches is evaluated through the condition numbers of the resulting systems and the corresponding GMRES/PGMRES iteration counts. Finally, we consider the deformation of a freely vibrating slanted plate to illustrate the performance of the proposed method on a representative fluid--plate interaction problem. All computations are performed in FreeFem++ \cite{MR3043640}. 

\subsection{Convergence Test}

We first verify the spatial and temporal accuracy of Alg.~\ref{alg:scur_dd_alg} by using manufactured solutions, following previous work \cite{BesabeLee2026}. The computational domains are
\begin{equation*}
    \Omega_f=[0,1]\times[0,1]\times[-1,0] \quad\text{and} \quad \Omega_p=[0,1]\times[0,1]\times\{0\},
\end{equation*}
and the forcing terms are chosen so that the following prescribed fluid velocity $\bm{u}=(u_1, u_2, u_3)^T$, fluid pressure $p$, plate displacement $w$, and auxialiary variable $z$ are satisfied: 
\begin{align*}
    u_1 &= \zeta\left(-\frac{\cos(2\pi x)}{4\pi}+\frac{\cos(4\pi x)}{16\pi}+\frac{3}{16\pi}\right)\sin^2(\pi y)\sin(2\pi y)\left(\frac{\pi}{2}\sin(\pi(z+1))\right)e^{-t}, \\
    u_2 &= 0, \\
    u_3 &= -\zeta\sin^2(\pi x)\sin(2\pi x)\sin^2(\pi y)\sin(2\pi y)\sin^2\left(\frac{\pi}{2}(z+1)\right)e^{-t},\\
    p &= 0, \\
    w &= \zeta\sin^2(\pi x)\sin(2\pi x)\sin^2(\pi y)\sin(2\pi y)e^{-t},\\
    z & = - \zeta\left(3\pi^2\sin(4\pi x) - 4\pi^2\sin^2(\pi x)\sin(2\pi x)\right)\sin^2(\pi y)\sin(2\pi y)e^{-t} \\
    &\quad -\zeta\sin^2(\pi x)\sin(2\pi x)\left(3\pi^2\sin(4\pi y) - 4\pi^2\sin^2(\pi y)\sin(2\pi y)\right)e^{-t}.
\end{align*}

The fluid velocity $\bm{u}$ and fluid pressure $p$ are discretized using the inf-sup stable Taylor-Hood $\mathbb{P}_2-\mathbb{P}_1$ element pair. The plate displacement $w$ and auxiliary variable $z$ are approximated using continuous $\mathbb{P}_2$ elements, while the Lagrange multiplier $\bm{g}$ is approximated using continuous $\mathbb{P}_1$ elements. For the spatial convergence, we fix the time step $\delta t = 1e-04$ and final time $T = 1e-03$, and refine $h_f$ and $h_p$. For the errors, we introduce the notation $e_{\Phi}^n = \Phi(t^n) - \Phi_h^n$ to be the difference between the exact solution $\Phi$ and the discrete solution $\Phi_h$ at time $t^n$. For brevity, the mesh sizes for the fluid and plate domains are taken to be the same, i.e, $h_f = h_p = h$. To prevent excessively large forcing terms associated with the biharmonic contribution, we scale the manufactured solution by choosing $\zeta = (60\pi^4)^{-1}$. 
All other parameters are set to 1. Additionally, we set the GMRES tolerance to $1e-08$.

In Fig.~\ref{fig:space_error}, we display the errors for the fluid and plate variables with respect to changes in the mesh size $h$. We note that the slopes match the optimal convergence rates, i.e., $O(h^3)$ for $\|e_u^N\|_{\Omega_f}, \|e_w^N\|_{\Omega_p}, \|e_z^N\|_{\Omega_p}$ and $O(h^2)$ for $\|e_u^N\|_{1,\Omega_f}, \|e_p^N\|_{\Omega_f}, \|e_w^N\|_{1,\Omega_p}, \|e_z^N\|_{1,\Omega_p}$. 

\begin{figure}[htb!]
    \centering
    \begin{subfigure}{0.35\linewidth}
        \includegraphics[width = \linewidth]{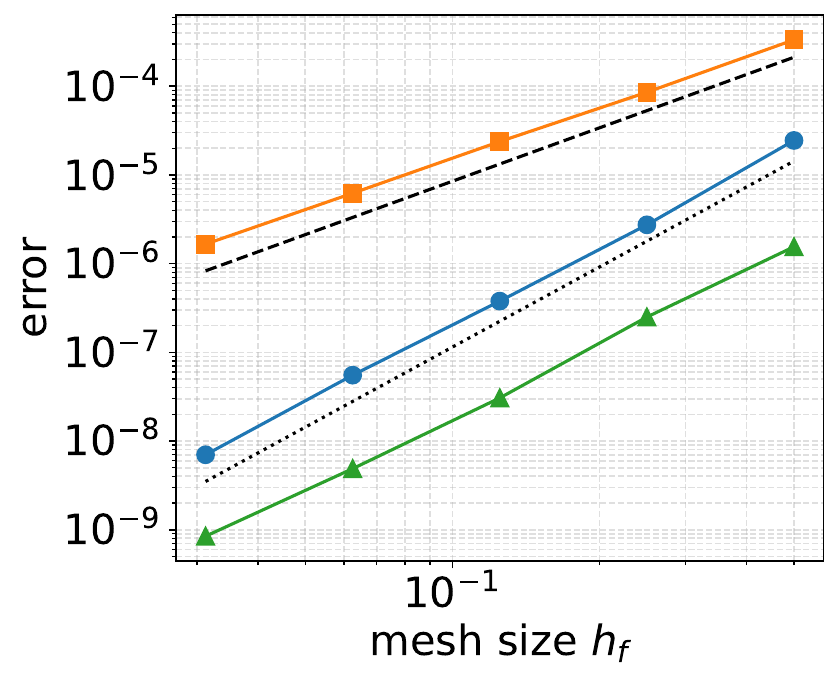}
    \end{subfigure}
    \begin{subfigure}{0.35\linewidth}
        \includegraphics[width = \linewidth]{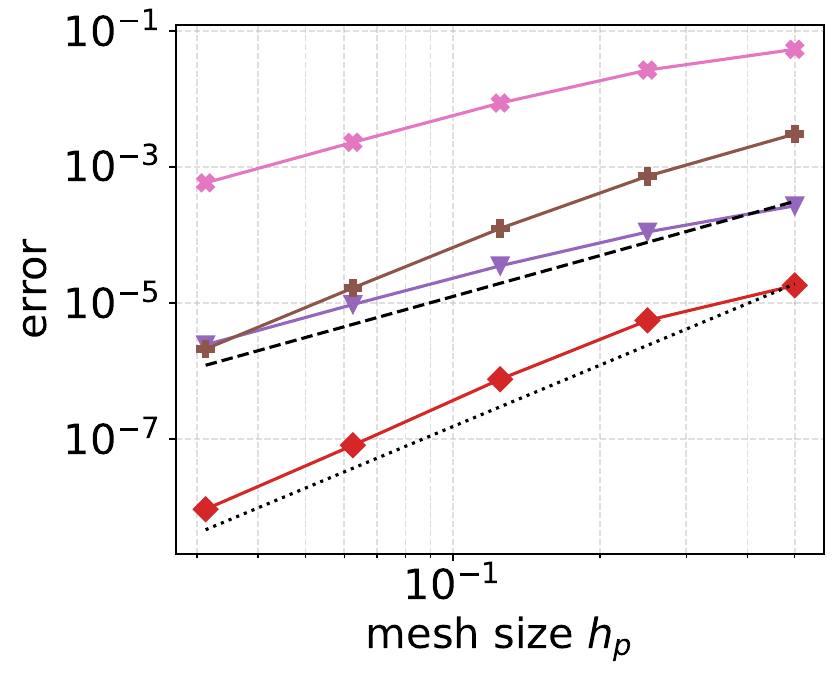}
    \end{subfigure}
    \begin{subfigure}{0.65\linewidth}
        \includegraphics[width = \linewidth]{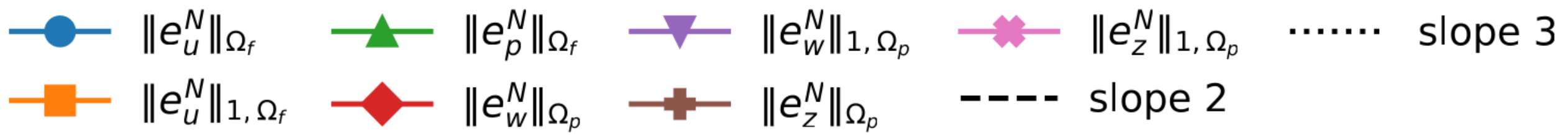}
    \end{subfigure}
    \caption{Spatial convergence: $L^2$ and $H^1$ error for fluid (left) and plate (right) variables for different mesh sizes with $\delta t = 1e-04$ and $T = 1e-03$.}
    \label{fig:space_error}
\end{figure}

Table~\ref{tab:preconditioner_comparison} compares the performance of the unpreconditioned Schur complement system with the proposed preconditioner, $\calmf = \calbf\calbf^T$, and the block Jacobi preconditioner as the mesh is refined. Without preconditioning, the condition number of the Schur complement matrix $\calsf$ increases rapidly with decreasing mesh size, leading to a dramatic growth in the number of GMRES iterations and demonstrating the poor scalability of the unpreconditioned system. In contrast, the proposed preconditioner significantly reduces the condition number, maintaining it at approximately $20$-$30$ across all mesh sizes. As a result, the number of PGMRES iterations remains nearly mesh independent, varying only between $32$ and $92$. Although the block Jacobi preconditioner also improves the conditioning relative to the unpreconditioned system, its condition number grows steadily under mesh refinement and requires substantially more GMRES iterations. These results demonstrate that the proposed preconditioner provides a considerably more robust and scalable solver for the Schur complement system than the block Jacobi approach.

\begin{table}[htb!]
    \centering
    \small
    \begin{tabular}{|c|c|c|cc|cc|}
    \hline
    & & &
    \multicolumn{2}{c|}{$\calmf=\widetilde{B}\widetilde{B}^T$} &
    \multicolumn{2}{c|}{Block Jacobi} \\
    \cline{4-7}
    $h$ & $\kappa(\calsf)$ & GMRES &
    $\kappa(\calmf^{-1}\calsf)$ & PGMRES &
    $\kappa(\calmf^{-1}\calsf)$ & PGMRES \\
    \hline
    $1/2$  & $1.76e+04$ & $34$    & $1.99e+01$ & $32$ & $3.50e+01$ & $36$ \\
    $1/4$  & $4.92e+04$ & $12033$ & $2.62e+01$ & $64$ & $7.18e+01$ & $344$ \\
    $1/8$  & $1.29e+05$ & $15145$ & $2.59e+01$ & $86$ & $9.16e+01$ & $413$ \\
    $1/16$ & $2.91e+05$ & $17478$ & $2.19e+01$ & $92$ & $1.28e+02$ & $205$ \\
    $1/32$ & $5.01e+05$ & $20067$ & $2.64e+01$ & $73$ & $1.45e+02$ & $145$ \\
    \hline
    \end{tabular}
    \caption{Spatial convergence: comparison of the proposed preconditioner $M=BB^T$ and the block Jacobi preconditioner for the Schur complement system. We report the condition number $\kappa(\calsf)$ of the Schur complement matrix $\calsf$ and the preconditioned systems, $\kappa(M^{-1}S)$, together with the corresponding average GMRES/PGMRES iteration counts for different mesh sizes $h$ with $\delta t = 1e-04$ and $T = 1e-03$.}
    \label{tab:preconditioner_comparison}
\end{table}

Figure~\ref{fig:wdot_u_comp} compares the computed plate velocity $\dot{w}$ with the normal component of the fluid velocity, $\mathbf{u}\cdot\mathbf{n}$, on the plate $\Omega_p$ at the final time instant $t = 1e-03$ with mesh size $h = 1/32$. The two solutions are visually indistinguishable over the entire interface, indicating that the kinematic coupling condition is satisfied. The rightmost panel displays the pointwise absolute difference between the two quantities, whose maximum value is approximately $10^{-8}$, consistent with the prescribed GMRES stopping tolerance.

\begin{figure}
    \centering
    \begin{subfigure}{0.34\textwidth}
        \includegraphics[width = \textwidth]{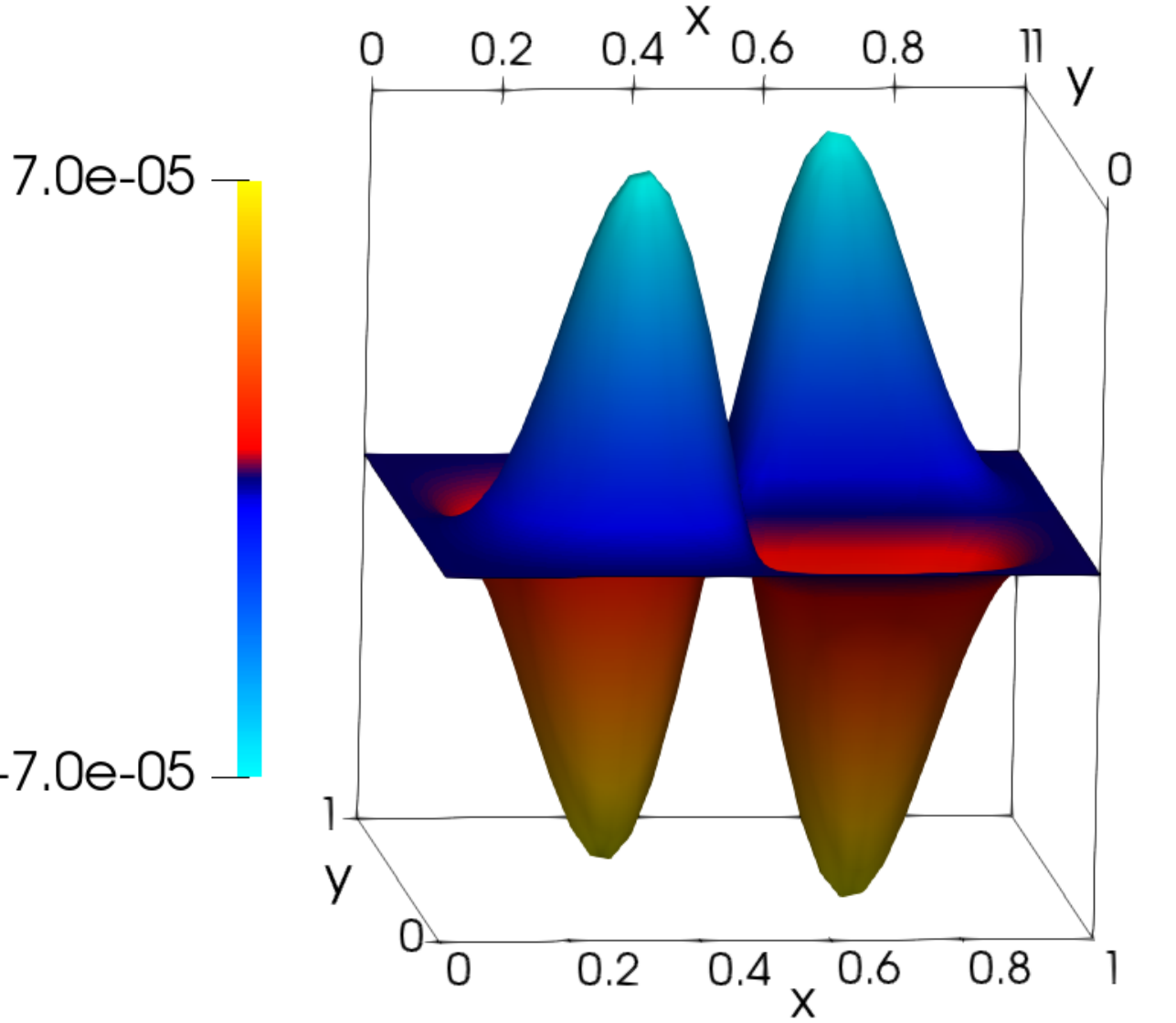}
    \end{subfigure}
    \begin{subfigure}{0.275\textwidth}
        \includegraphics[width = \textwidth]{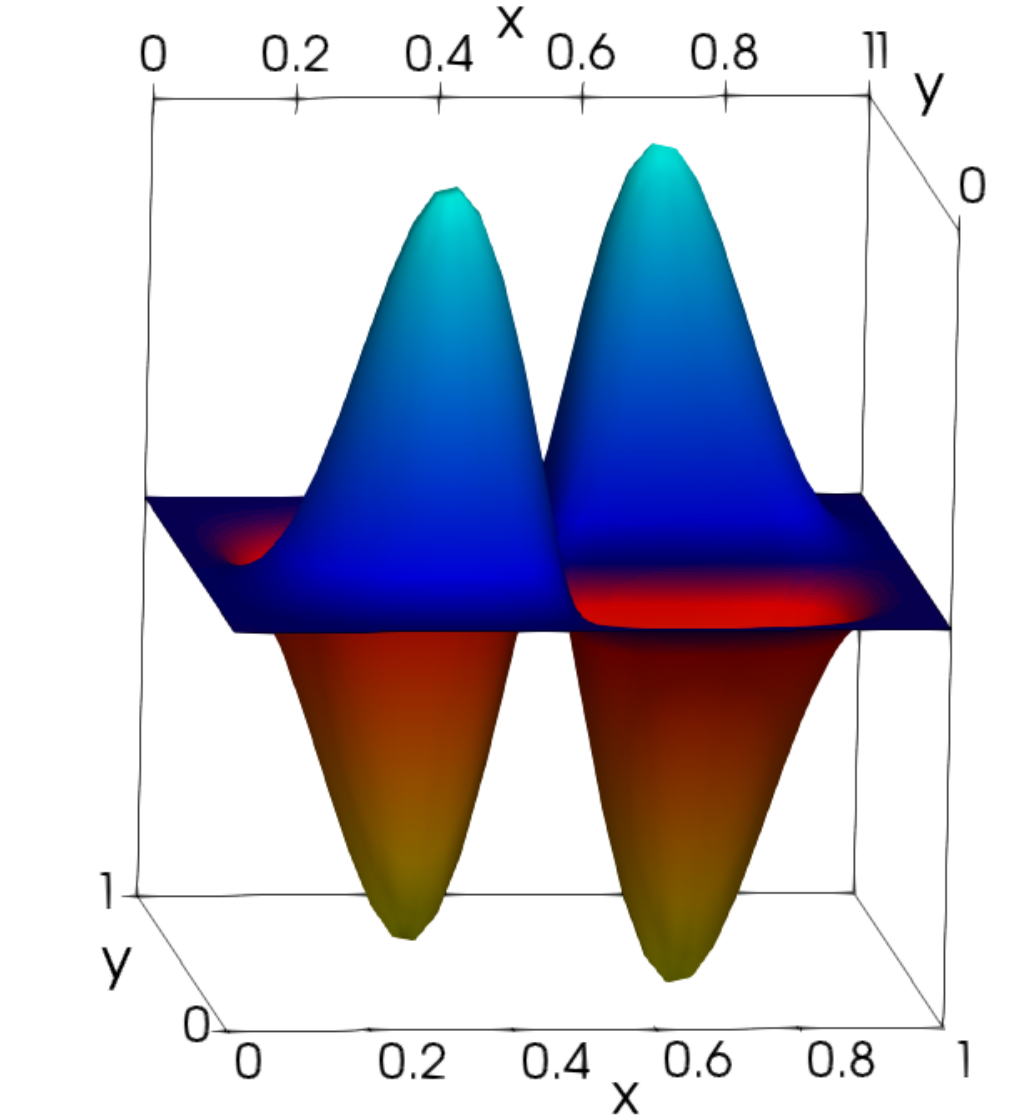}
    \end{subfigure}
    \begin{subfigure}{0.34\textwidth}
        \includegraphics[width = \textwidth]{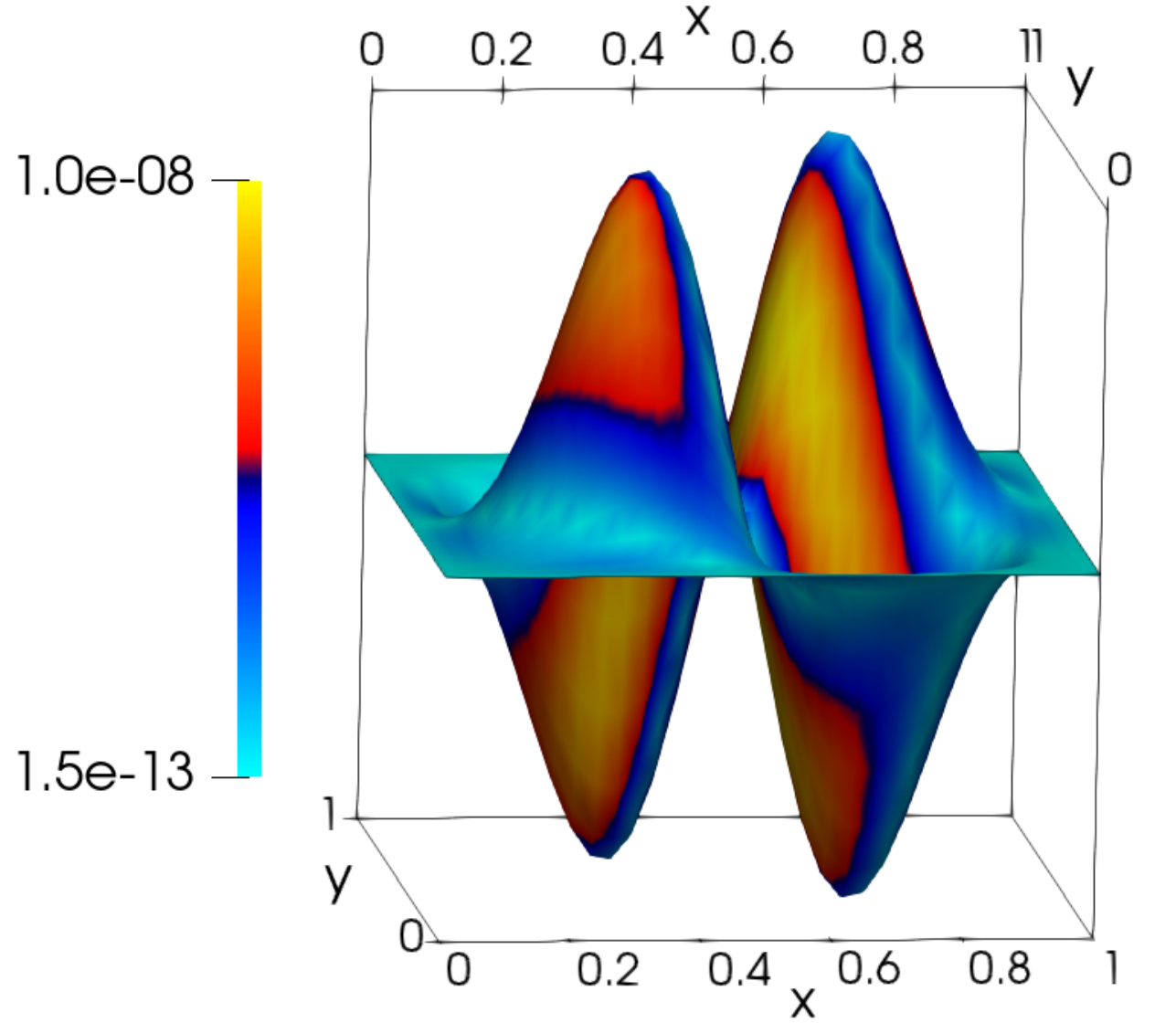}
    \end{subfigure}
    \caption{Convergence test: velocity of the plate $\dot{w}$ (left), normal component of the velocity $\bm{u}\cdot\bm{n}$ on the plate $\Omega_p$ (center), and their absolute difference (right) at $t = 1e-03$ with $h_p = h_f = 1/32$. Note that the warping corresponds to the computed $w$.}
    \label{fig:wdot_u_comp}
\end{figure}

To evaluate the temporal convergence of the proposed method, we fix the spatial mesh to $h_f = h_p = h=1/16$, set the final time to $T=1$, and successively refine the time step $\delta t$. As reported in \cite{Geredeli_Kunwar_Lee2024}, the temporal error remains dominated by the spatial discretization error unless the contribution of the latter is reduced. Following the strategy in \cite{Geredeli_Kunwar_Lee2024}, we set the plate density to $\rho_p=10^5$ in the plate equation \eqref{eq:plate_eq} while leaving all remaining parameters unchanged as in the spatial convergence test. This modification helps isolate the temporal discretization error. Figure~\ref{fig:time_error} shows that the errors for the fluid and plate variables decrease linearly with respect to $\delta t$, in agreement with the theoretical first-order accuracy of the backward Euler scheme.

\begin{figure}[htb!]
    \centering
    \begin{subfigure}{0.35\linewidth}
        \includegraphics[width = \linewidth]{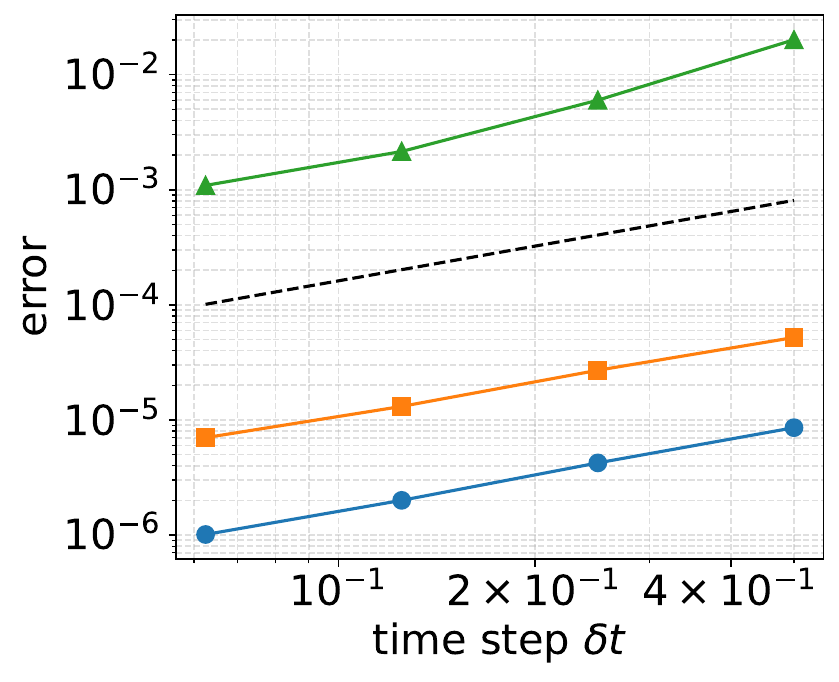}
    \end{subfigure}
    \begin{subfigure}{0.35\linewidth}
        \includegraphics[width = \linewidth]{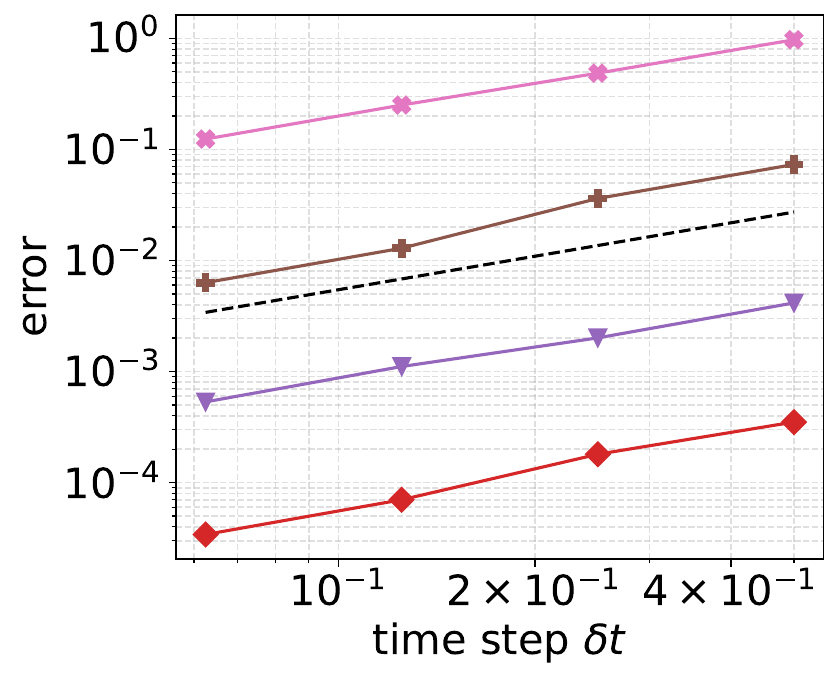}
    \end{subfigure}
    \begin{subfigure}{0.6\linewidth}
        \includegraphics[width = \linewidth]{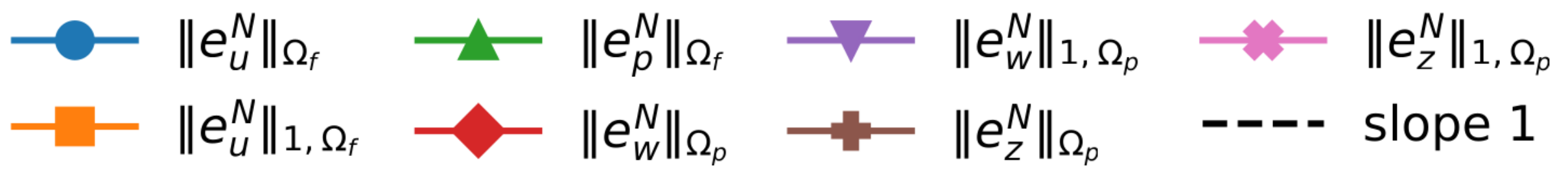}
    \end{subfigure}
    \caption{Temporal convergence: $L^2$ and $H^1$ error for fluid (left) and plate (right) variables for different time steps with $h_f = h_p = h = 1/16$ and $\omega = 1e5$.}
    \label{fig:time_error}
\end{figure}

\subsection{Free Vibration of a Slanted Plate}

To further assess the performance of the proposed partitioned algorithm, we consider the free vibration of a slanted elastic plate interacting with a viscous incompressible fluid, inspired by \cite{Nguyen2021}. The computational domain consists of a 3D fluid region with an inclined upper boundary, while the plate occupies the slanted interface. The fluid domain is given by
\begin{equation*}
    \Omega_f = \{(x, y, z): 0\leq x\leq 1, 0\leq y\leq 1, 0 \leq z\leq h(x)\}
\end{equation*}
where
\begin{equation*}
    h(x) = h_\text{min} + x\tan\theta,
\end{equation*}
where $h_\text{min} = 1 - \theta$, and we set $\theta = \pi/6$. Note that if $\theta = 0$, then the plate becomes horizontal. Consequently, the plate occupies the inclined surface
\begin{equation*}
    \Omega_p = \{(x, y, h(x)): 0\leq x,y\leq 1\}.
\end{equation*}
The fluid is initially at rest, and no external body forces are applied to the fluid. The plate is prescribed a nonzero initial displacement with zero initial velocity,
\begin{equation*}
    w = w_0 = A\sin(2\pi x)\sin(2\pi y), \quad \partial_{t}w = w_{t0} = 0
\end{equation*}
where $A = 1e-02$. We use a mesh size of $h_f = h_p = h = 1/16$, a time step of $\delta t = 1e-3$, and an end time of $T = 1e-1$. Additionally, we fix the following physical parameters:
\begin{equation*}
    \rho_f = 1, \quad \nu_f = 1e-3, \quad \rho = 1e-3,
\end{equation*}
while the plate density $\rho_p$ is varied to explore the performance of the proposed algorithm in the added-mass regime. As $\rho_p$ decreases, the added-mass effect becomes increasingly noticeable, making the coupled fluid-plate system more challenging to solve. This allows us to assess the robustness of the proposed preconditioner over a range of density ratios between the fluid and plate densities.




We present the displacement $w$ of the plate normal to $\Omega_p$ at time $t = 0.1$ s when we set $\rho_p = 1$ in Fig.~\ref{fig:w_test2}. 

\begin{figure}[htb!]
    \centering
    \includegraphics[width=0.5\linewidth]{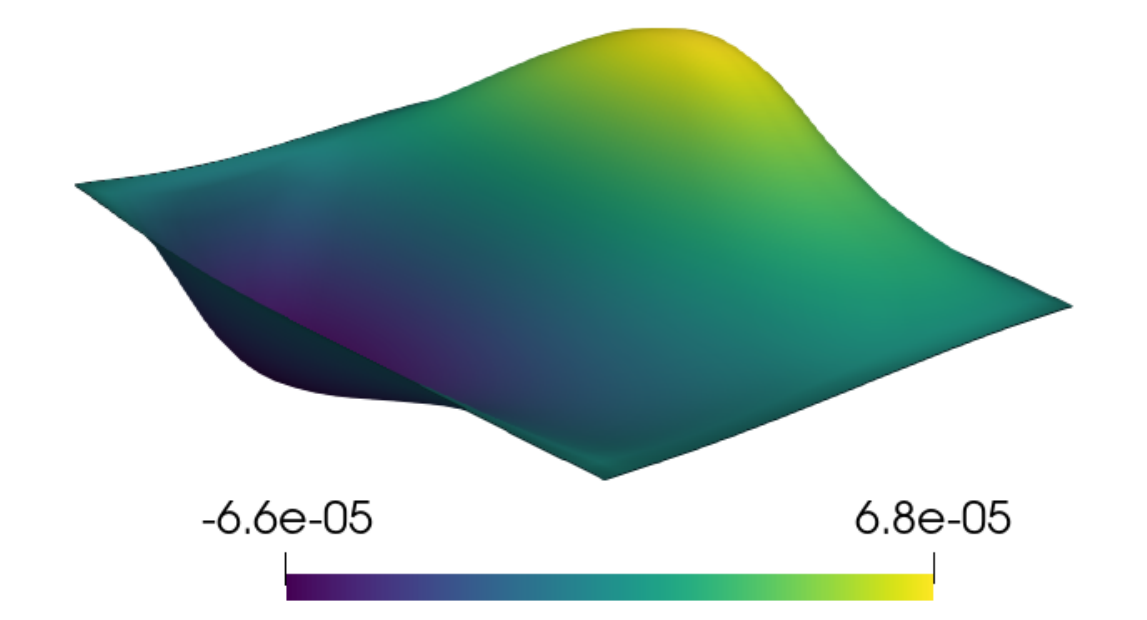}
    \caption{Free vibrating plate: displacement $w$ of the plate normal to $\Omega_p$ at time $t = 0.1$ s when $\rho_p = 1e-2.$ Note that the deformation and the colorbar are according to the values of $w$ and the the plate is rotated for convenience.}
    \label{fig:w_test2}
\end{figure}

In Fig.~\ref{fig:exp2_wdot_u3}, we report the plate velocity and the normal component of the fluid velocity at the interface for varying plate densities. As the plate density decreases, the density ratio between the fluid and structure approaches the added-mass regime, producing increasingly strong fluid–structure coupling. Nevertheless, the computed fluid and structural velocities remain visually indistinguishable across all cases, indicating that the proposed Schur complement formulation accurately enforces the kinematic interface condition independently of the density ratio. Although the interface error increases for the smallest values of $\rho_p$, it remains several orders of magnitude smaller than the magnitude of the solution, demonstrating the robustness of the proposed formulation even in challenging added-mass regimes.

\begin{figure}[htb!]
    \centering
    \begin{tabular}{cccc}
         & $\dot{w}(t)$ & $\bm{u}\cdot\bm{n}$ & $\abs{\dot{w}( t) - \bm{u}\cdot\bm{n}|_{\Omega_p}}$ \\
        $\rho_p = 1$ & \includegraphics[align=c, width = 0.25\textwidth]{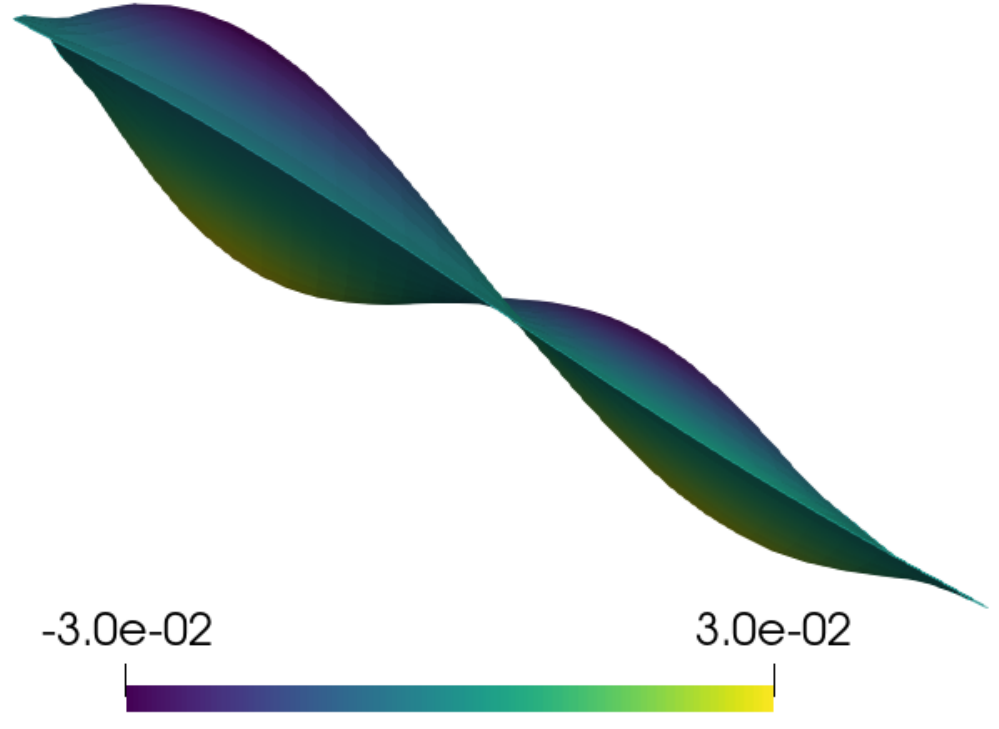} & \includegraphics[align=c, width = 0.25\textwidth]{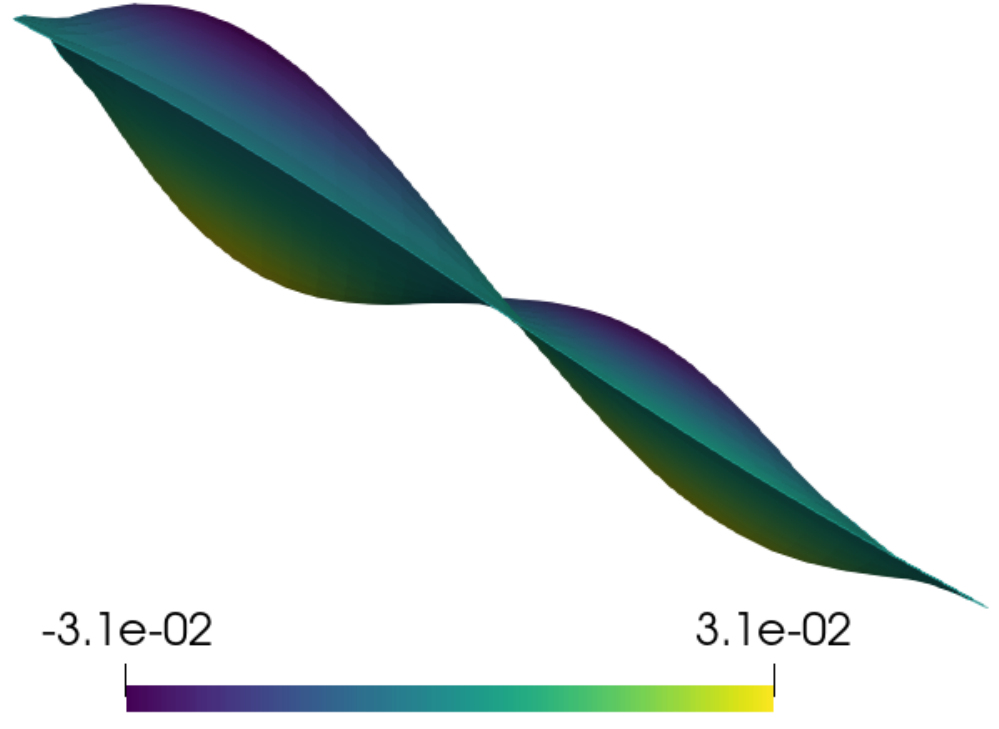} & \includegraphics[align=c, width = 0.25\textwidth]{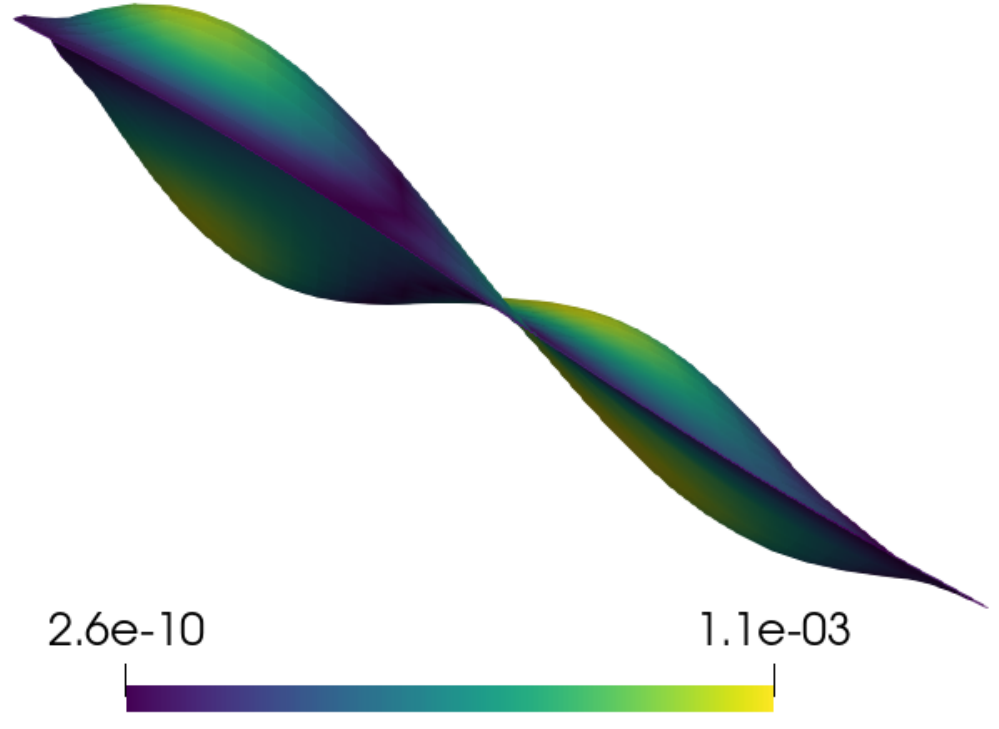} \\
        $\rho_p = 1e-1$ & \includegraphics[align=c, width = 0.25\textwidth]{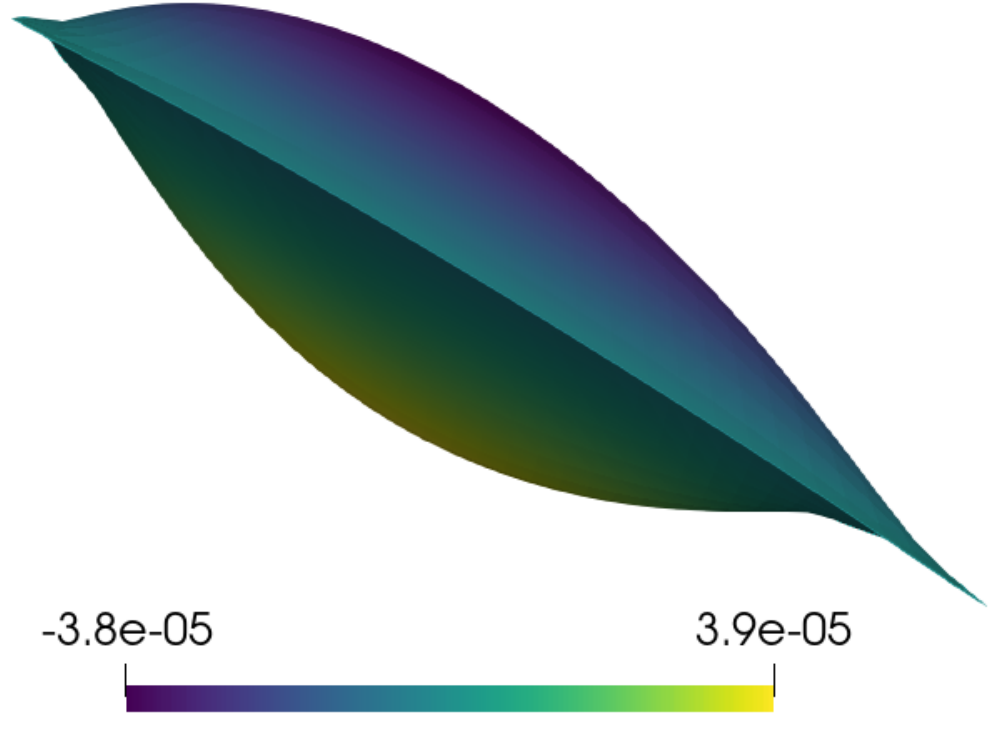} & \includegraphics[align=c, width = 0.25\textwidth]{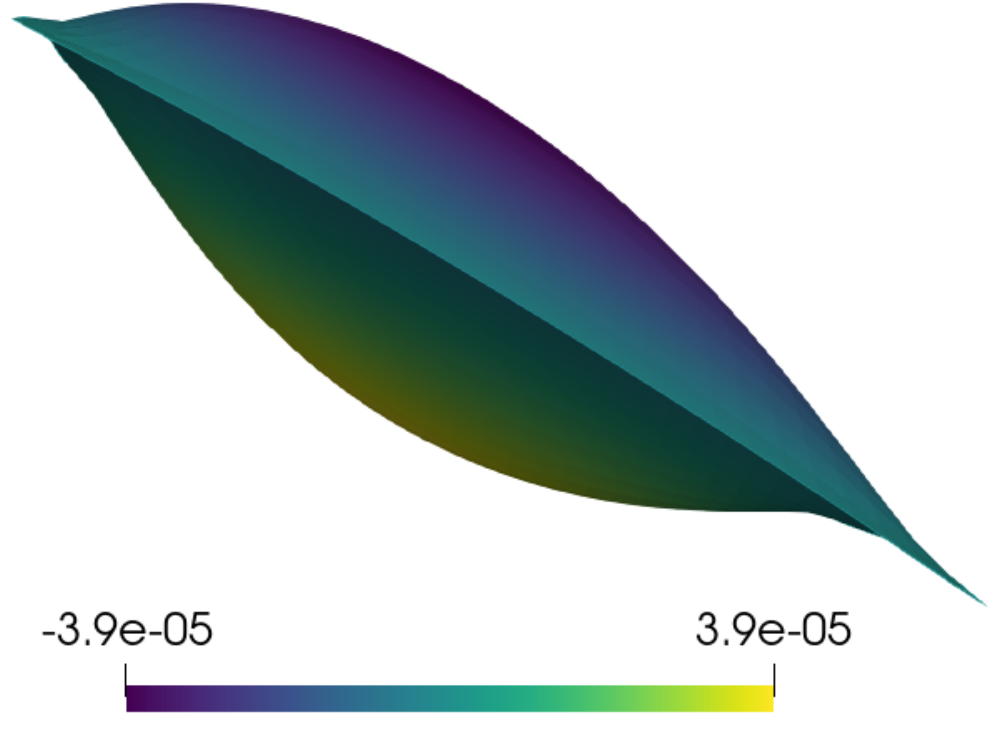} & \includegraphics[align=c, width = 0.25\textwidth]{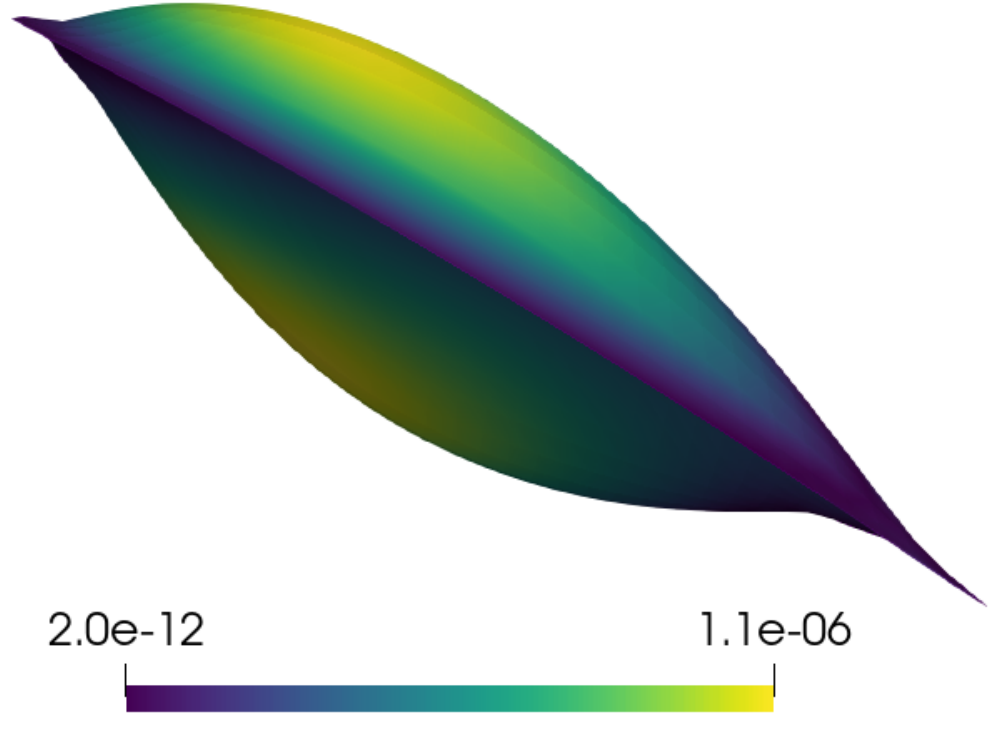} \\
        $\rho_p = 1e-2$ & \includegraphics[align=c, width = 0.25\textwidth]{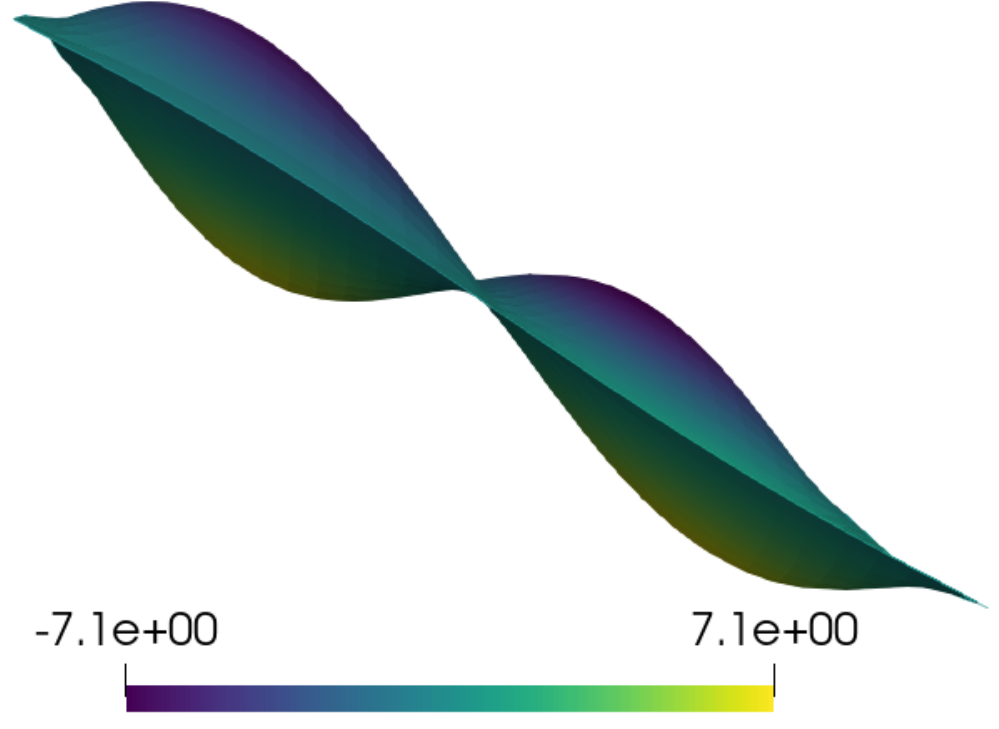} & \includegraphics[align=c, width = 0.25\textwidth]{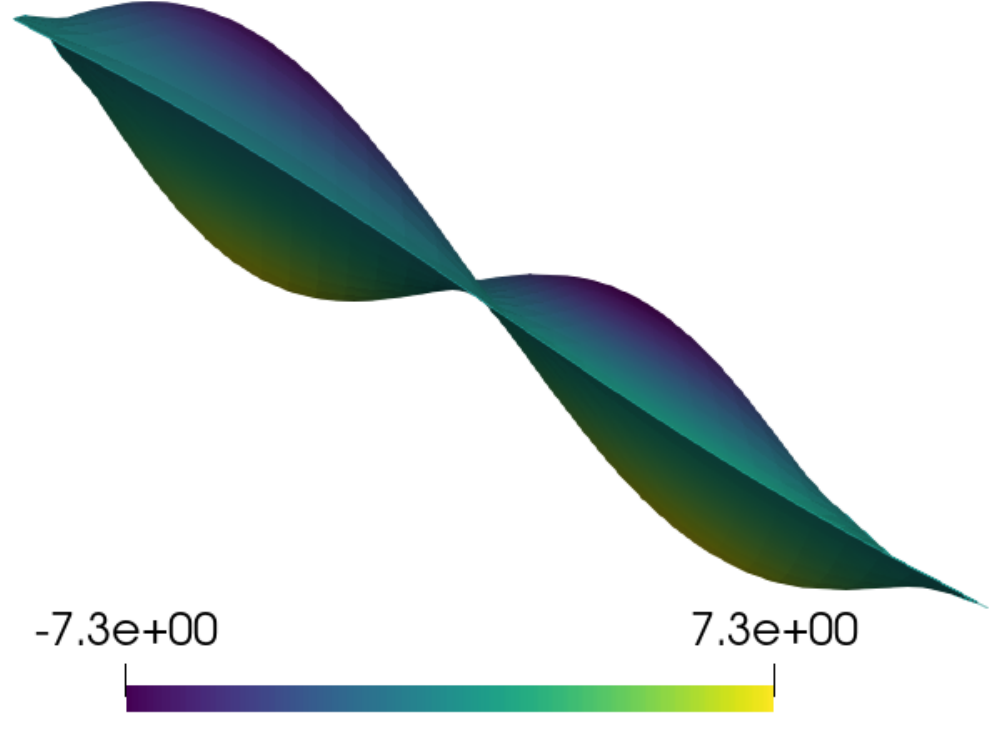} & \includegraphics[align=c, width = 0.25\textwidth]{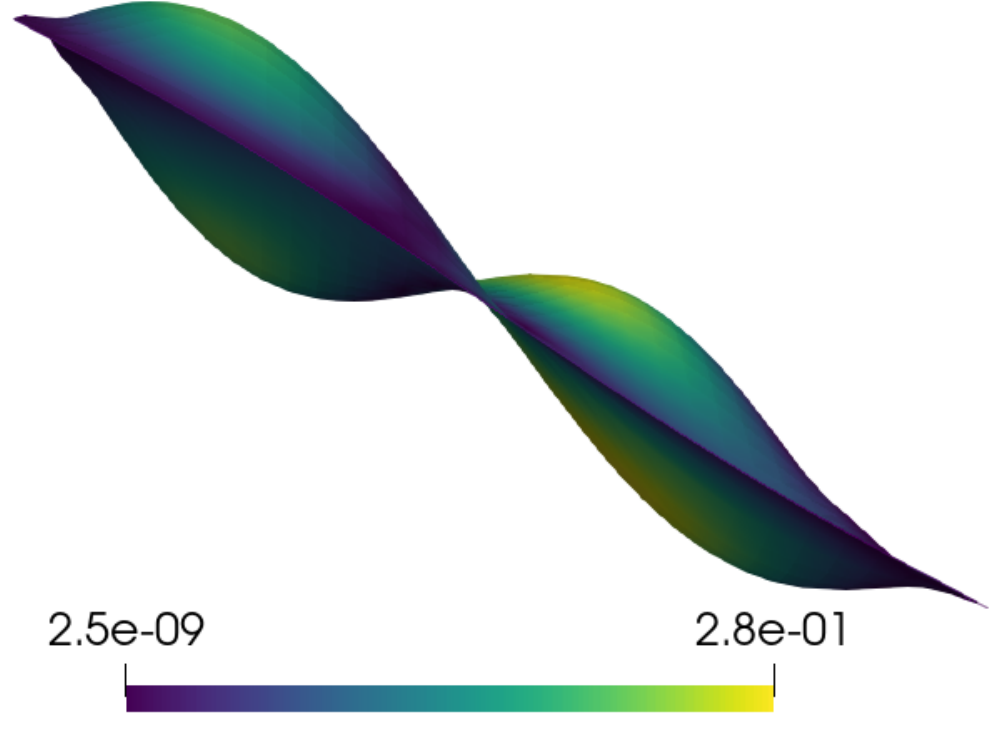} \\
        $\rho_p = 1e-3$ & \includegraphics[align=c, width = 0.25\textwidth]{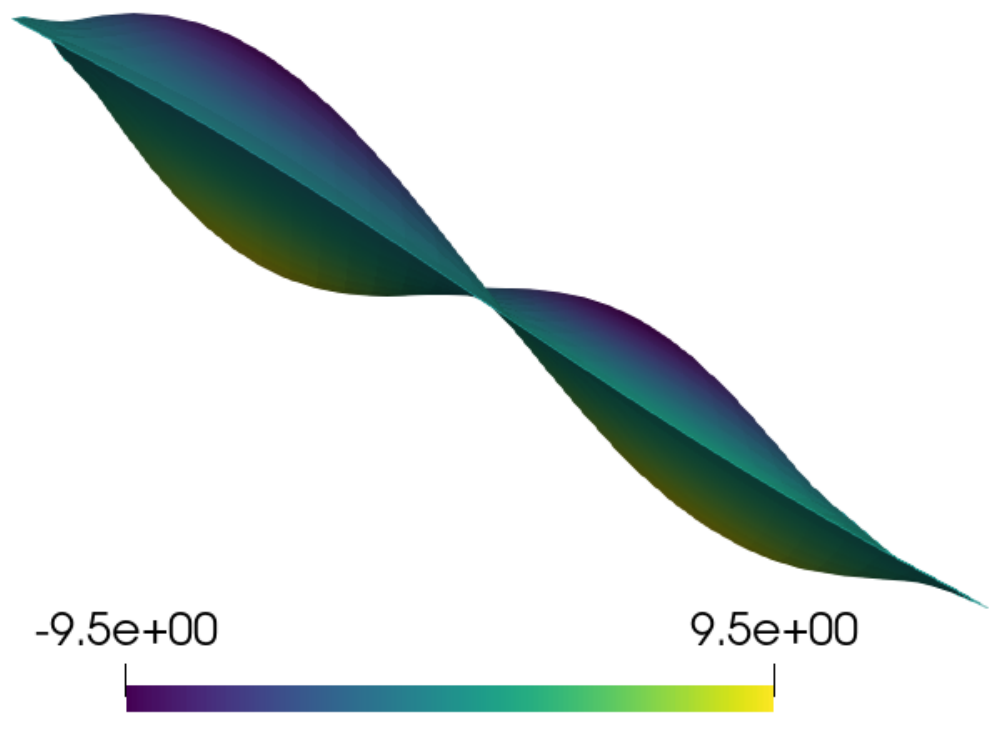} & \includegraphics[align=c, width = 0.25\textwidth]{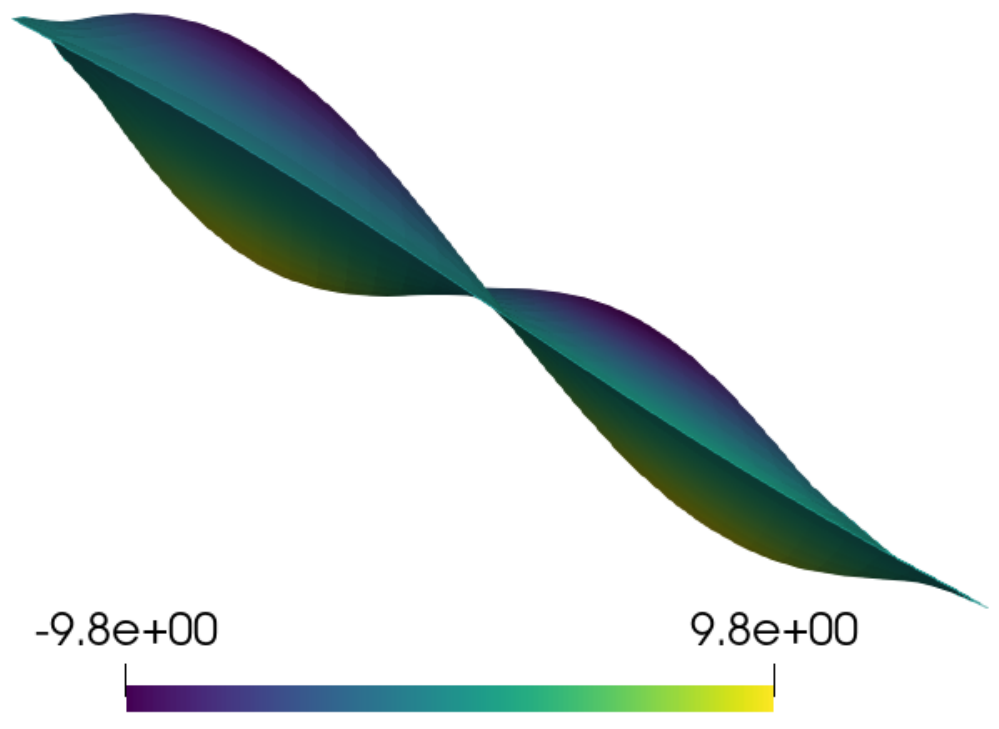} & \includegraphics[align=c, width = 0.25\textwidth]{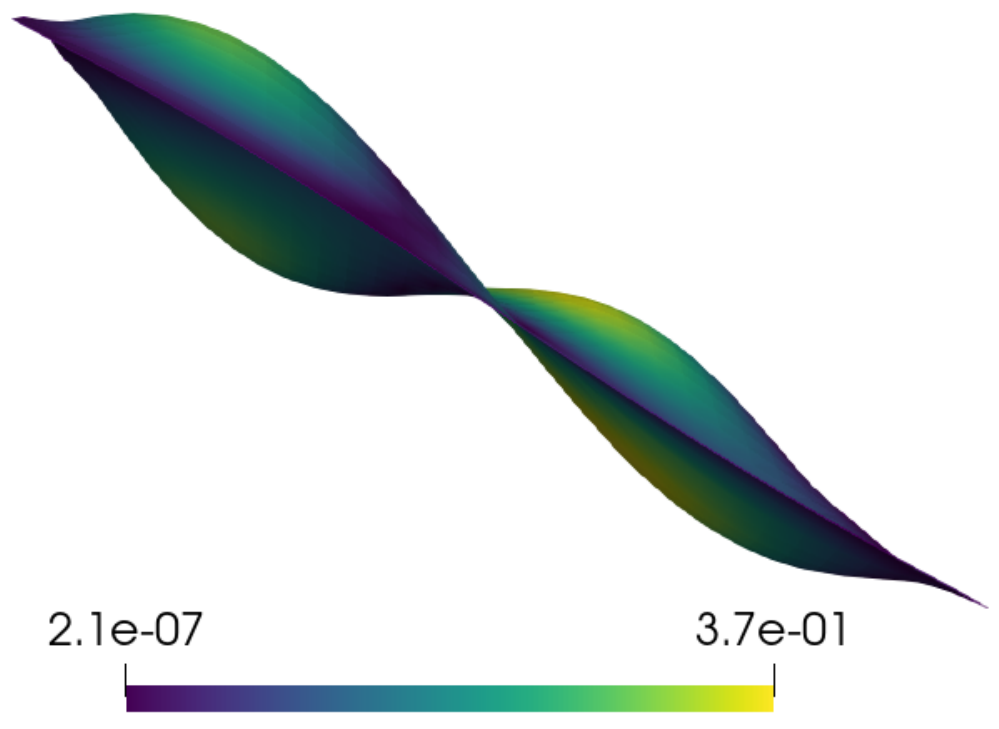} \\
    \end{tabular}
    \caption{Free vibrating plate: plate velocity $\dot{w}$ (left), normal component of the fluid velocity $\bm{u}$ at the plate $\Omega_p$ (center), and their absolute difference (right) at $t = 0.1$ across different values of $\rho_p = 1, 1e-1, 1e-2, 1e-3$ with $\rho_f = 1$ fixed.}
    \label{fig:exp2_wdot_u3}
\end{figure}

We report in Fig.~\ref{fig:test2_energy_err} the evolution of the energy defined by
\begin{equation*}
    E_\text{total} \coloneqq \frac{\rho_f}{2}\int_{\Omega_f} \abs{\bm{u}}^2\,d\Omega_f + \frac{\rho_p}{2}\int_{\Omega_p} \abs{\dot{w}}^2\,d\Omega_p + \frac{\rho}{2}\int_{\Omega_p} \abs{\nabla \dot{w}}^2\,d\Omega_p + \frac{1}{2}\int_{\Omega_p}\abs{z}^2\,d\Omega_p,
\end{equation*}
together with the relative interface mismatch we define by
\begin{equation*}
    \frac{\|\bm{u}(t)\cdot\bm{n} - \dot{w}(t)\|_{\Omega_p}}{\|\dot{w}(t)\|_{\Omega_p}}.
\end{equation*}
As expected, the energy decreases monotonically due to viscous dissipation in the fluid. Moreover, the rate of decay depends strongly on the plate density, with lighter plates transferring energy more rapidly to the surrounding fluid. The relative interface mismatch remains small, i.e., within 3\% to 6\%, throughout the simulation for all density ratios considered, further confirming that the interface conditions are accurately maintained during the evolution.

\begin{figure}
    \centering
    \begin{subfigure}{0.48\linewidth}
        \centering
        \includegraphics[width = \linewidth]{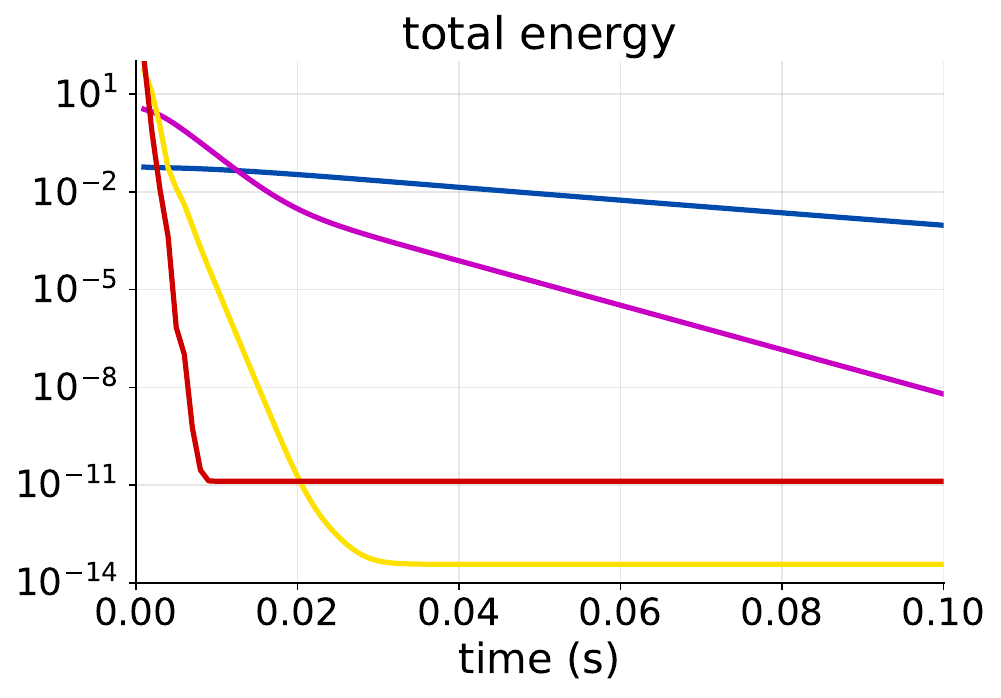}
    \end{subfigure}
    \begin{subfigure}{0.48\linewidth}
        \centering
        \includegraphics[width = \linewidth]{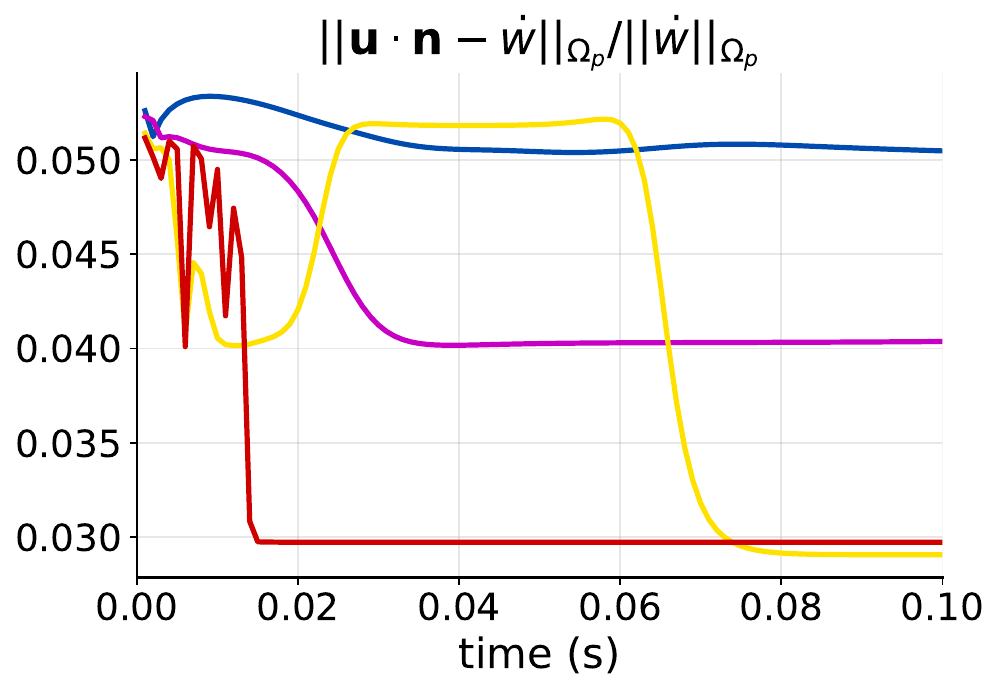}
    \end{subfigure}
    \begin{subfigure}{0.7\linewidth}
        \centering
        \includegraphics[width = \linewidth]{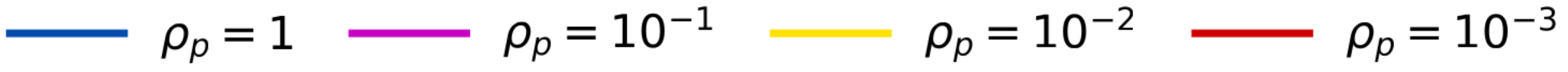}
    \end{subfigure}
    \caption{Free vibrating plate: total mechanical energy (left) and relative interface residual (right) for different values of the plate density $\rho_p$.}
    \label{fig:test2_energy_err}
\end{figure}

Lastly, Fig.~\ref{fig:test2_gmres} compares the average GMRES iteration counts required to solve the Schur complement system for different plate densities. As the plate density decreases, the coupled problem enters a progressively stronger added-mass regime, which is well known to possess difficulties for partitioned fluid–structure algorithms \cite{Causin-Nobile2005}. Nevertheless, the proposed interface preconditioner maintains nearly constant iteration counts over the entire range of density ratios considered. In contrast, the block Jacobi preconditioner consistently requires substantially more iterations, while the unpreconditioned Schur complement system remains significantly more expensive to solve. These results indicate that the proposed preconditioner is largely insensitive to the strength of the fluid–structure coupling and effectively mitigates the deterioration in convergence typically associated with increasingly severe added-mass effects.

\begin{figure}
    \centering
    \includegraphics[width=0.5\linewidth]{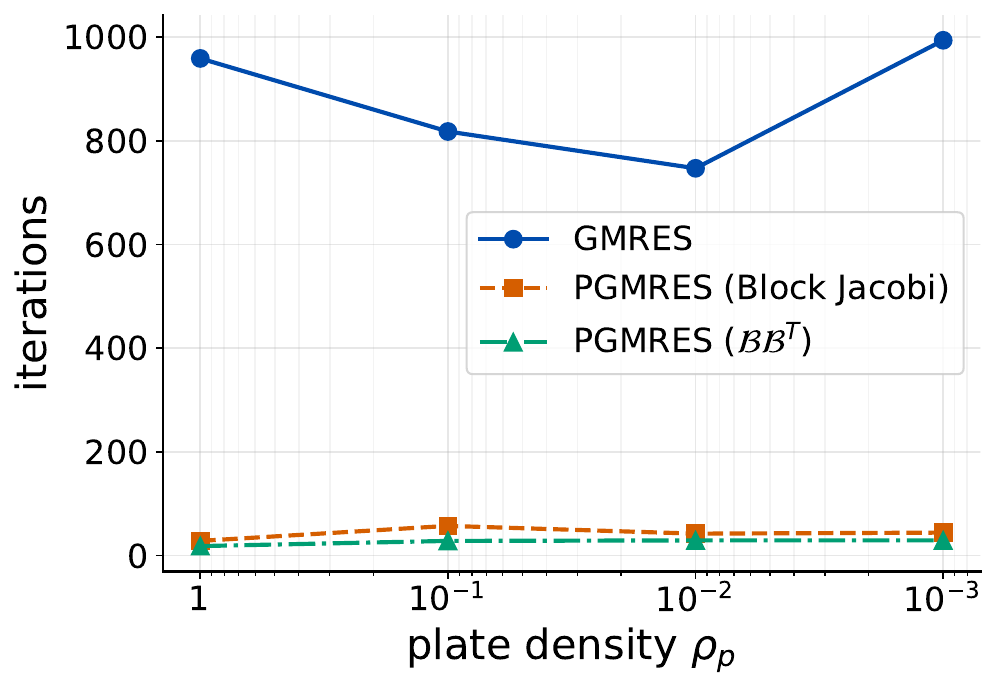}
    \caption{Free vibrating plate: Average GMRES/PGMRES iterations to solve \eqref{eq:schur_eq} without a preconditioner, with }
    \label{fig:test2_gmres}
\end{figure}

\section{Conclusions} \label{sec:conclusions}

We have presented a Schur complement domain decomposition method for a three-dimensional fluid–two-dimensional plate interaction problem based on a mixed finite element formulation with Lagrange multipliers. The proposed formulation decouples the fluid and structural subproblems while preserving the strong interface coupling and naturally admits a matrix-free implementation. We analyzed the conditioning of the Schur complement system matrix, identified the interface operator as the dominant source of its ill-conditioning, and proposed an interface-based preconditioner with a corresponding theoretical conditioning estimate. Numerical experiments confirmed the predicted convergence rates and demonstrated that the proposed preconditioner consistently outperforms both the unpreconditioned Schur complement formulation and a block Jacobi preconditioner in terms of conditioning and GMRES iteration counts, while remaining robust for challenging added-mass regimes. These results indicate that the proposed approach provides an efficient and mathematically justified framework for partitioned simulation of three-dimensional fluid–plate interaction problems.

\section*{Acknowledgements} 
{Hyesuk Lee was partially supported by the NSF under grant numbers DMS-2207971 and DMS-2513073.} {This research used in part resources on the Palmetto Cluster at Clemson University under National Science Foundation awards MRI 1228312, II NEW 1405767, MRI 1725573, and MRI 2018069.}

\bibliography{FSI} 

\end{document}